\documentclass[a4paper]{amsart}

\usepackage[dvipdfmx]{hyperref}
\usepackage{enumitem}

\usepackage{mathrsfs}

\numberwithin{equation}{section}

\newtheorem{theorem}{Theorem}[section]
\newtheorem{proposition}[theorem]{Proposition}
\newtheorem{lemma}[theorem]{Lemma}
\newtheorem{corollary}[theorem]{Corollary}

\theoremstyle{definition}

\newtheorem{assumption}[theorem]{Assumption}

\theoremstyle{remark}
\newtheorem{remark}[theorem]{Remark}

\newcommand{\tref}[2][]
	{\ifx#1""{Theorem~\textup{\ref{#2}}}%
	\else{Theorems~\textup{\ref{#1}} and~\textup{\ref{#2}}}\fi}
\newcommand{\trefs}[4][]
	{\ifx#1""{Theorems~\textup{\ref{#2}},~\textup{\ref{#3}} and~\textup{\ref{#4}}}%
	\else{Theorems~\textup{\ref{#1}},~\textup{\ref{#2}},~\textup{\ref{#3}} and~\textup{\ref{#4}}}\fi}
\newcommand{\pref}[2][]
	{\ifx#1""{Proposition~\textup{\ref{#2}}}%
	\else{Propositions~\textup{\ref{#1}} and~\textup{\ref{#2}}}\fi}
\newcommand{\prefs}[4][]
	{\ifx#1""{Propositions~\textup{\ref{#2}},~\textup{\ref{#3}} and~\textup{\ref{#4}}}%
	\else{Propositions~\textup{\ref{#1}},~\textup{\ref{#2}},~\textup{\ref{#3}} and~\textup{\ref{#4}}}\fi}
\newcommand{\cref}[2][]
	{\ifx#1""{Corollary~\textup{\ref{#2}}}%
	\else{Corollaries~\textup{\ref{#1}} and~\textup{\ref{#2}}}\fi}
\newcommand{\lref}[2][]
	{\ifx#1""{Lemma~\textup{\ref{#2}}}%
	\else{Lemmas~\textup{\ref{#1}} and~\textup{\ref{#2}}}\fi}
\newcommand{\lrefs}[4][]
	{\ifx#1""{Lemma~\textup{\ref{#2}}, \textup{\ref{#3}} and~\textup{\ref{#4}}}%
	\else{Lemmas~\textup{\ref{#1}}, \textup{\ref{#2}}, \textup{\ref{#3}} and \textup{\ref{#4}}}\fi}
\newcommand{\rref}[2][]
	{\ifx#1""{Remark~\textup{\ref{#2}}}%
	\else{Remark~\textup{\ref{#1}} and~\textup{\ref{#2}}}\fi}
\newcommand{\iref}[2][]
	{\ifx#1""{\textup{(\ref{#2})}}%
	\else{\textup{(\ref{#1})}~and~\textup{(\ref{#2})}}\fi}
\newcommand{\aref}[2][]
	{\ifx#1""{Assumption~\textup{\ref{#2}}}%
	\else{Assumption~\textup{\ref{#1}} and~\textup{\ref{#2}}}\fi}
\newcommand{\secref}[2][]
	{\ifx#1""{Section~\textup{\ref{#2}}}%
	\else{Section~\textup{\ref{#1}} and~\textup{\ref{#2}}}\fi}

\newcommand{\RealNum}{\mathbf{R}}
\newcommand{\NaturalNum}{\mathbf{N}}
\newcommand{\Sym}[2]{\mathrm{Sym}_{#1}(#2)}
\newcommand{\Mat}[2]{\mathrm{Mat}_{#1}(#2)}
\newcommand{\idMat}{I}
\newcommand{\probSp}{\Omega}
\newcommand{\sigmaField}{\mathcal{F}}
\newcommand{\prob}{\boldsymbol{P}}
\newcommand{\expect}{\boldsymbol{E}}
\newcommand{\CM}{\mathfrak{H}}

\newcommand{\SobSp}[3][]{\mathcal{D}^{{#2},{#3}}}

\newcommand{\indicator}[1]{\mathtt{1}_{#1}}

\begin{document}

\title{Malliavin Calculus for Non-colliding Particle Systems}

\author[N. Naganuma]{Nobuaki Naganuma}
\address[N. N.]{Graduate School of Engineering Science, Osaka University, Toyonaka, Osaka 560-8531, Japan}
\email{naganuma@sigmath.es.osaka-u.ac.jp}

\author[D. Taguchi]{Dai Taguchi}
\address[D. T.]{Graduate School of Engineering Science, Osaka University, Toyonaka, Osaka 560-8531, Japan}
\email{dai.taguchi.dai@gmail.com}


\date{December 26th, 2018}

\begin{abstract}
In this paper, we use Malliavin calculus to show the existence and continuity of density functions of $d$-dimensional non-colliding particle systems such as hyperbolic particle systems and Dyson Brownian motion with smooth drift. For this purpose, we apply results proved by Florit and Nualart (1995) and Naganuma (2013) on locally non-degenerate Wiener functionals.
\end{abstract}

\maketitle

\section{Introduction}
\subsection{Background}

In the theory of stochastic differential equations (SDEs),
the existence and regularity/smoothness of density functions
with respect to the Lebesgue measure of solutions of SDEs
is a major research topic for which there are many results and methodologies of study.
Let us comment on
the approach from parabolic partial differential equations (PDEs)
and the approach from stochastic analysis.

Regarding the approach from parabolic PDEs,
a fundamental solution of a PDE is known to exist
if its coefficients are bounded and H\"older continuous
and if its diffusion coefficient is uniformly elliptic
(see Friedman \cite{Friedman1964}).
The fundamental solution is a density function of a solution
to the corresponding SDE by the Feynman--Kac formula.
The idea of the proof is based on Levi's parametrix method (perturbation of the drift),
which has been extended to a solution of an SDE with an $L^p$-valued drift coefficient (Portenko \cite{Portenko1990})
and a path-dependent drift coefficient (Kusuoka \cite{Kusuoka2017} and Makhlouf \cite{Makhlouf2016}).
The parametrix method leads to the differentiability (resp.\ H\"older continuity)
of the density function with respect to the initial variable (resp.\ terminal variable).

On the other hand, as an approach from stochastic analysis,
Malliavin calculus is a powerful tool,
and it is well-known that, under the H\"ormander condition,
a solution of an SDE with infinitely differentiable coefficients
has a smooth density function.
Because there exists a criterion that
non-degenerate Wiener functionals in the Malliavin sense
admit smooth density functions with respect to the Lebesgue measure,
we obtain the result by showing that the solution is non-degenerate.
For general theory of Malliavin calculus and applications for solutions of SDEs,
see \cite{MatsumotoTaniguchi2017,Shigekawa2004,Nualart2006,IkedaWatanabe1989}.
However, the criterion cannot be applied to solutions of SDEs with singular coefficients.
A squared Bessel process is a typical example of such SDEs;
its diffusion coefficient is singular at the origin although the coefficient is locally smooth.
Naganuma \cite{Naganuma2013} proposed an approach to access the squared Bessel process. He refined the notion of the local non-degeneracy of Wiener functionals introduced by Florit and Nualart \cite{FloritNualart1995}
and showed that solutions of squared-Bessel-type SDEs (and therefore of Bessel-type SDEs)
admit continuous density functions (see \cite[Theorem~2.2]{Naganuma2013}).
Note that inverse moments of the processes play a crucial role in the argument.
As another approach, De Marco \cite{DeMarco2011} showed the local existence of smooth density functions
of solutions of SDEs if their coefficients are locally smooth.

In this paper, we consider non-colliding particle systems of Dyson type and show
that it admits a continuous density function.
A typical example of such a system is the $\beta$-Dyson Brownian motion,
which describes the dynamics of non-colliding Brownian particles.
More precisely, for $d\geq 2$ and $T>0$,
the $d$-dimensional Dyson Brownian motion
$X=\{X(t)=(X_1(t),\dots,X_d(t))^\top\}_{0\leq t\leq T}$
with a parameter $\beta\geq 1$ is defined by a unique solution of an SDE
\begin{align}\label{eq_DysonBM}
	\left\{
		\begin{aligned}
			dX_i(t)
			&=
				\frac{\beta}{2}
				\sum_{k;k\neq i}
					\frac{dt}{X_i(t)-X_k(t)}
				+
				dW_i(t),
			\quad
			i=1,\dots,d,\\
			X(0)
			&=
				\bar{x},
		\end{aligned}
	\right.
\end{align}
where $\bar{x}$ is a deterministic initial condition belonging to
$
	\Delta_d
	=
	\{
		(x_1,\dots,x_d)^\top\in\RealNum^d;
		x_1<\dots<x_d
	\}
$
and
$W=\{W(t)=(W_1(t),\dots,W_d(t))^\top\}_{0\leq t\leq T}$
is a $d$-dimensional standard Brownian motion
on the canonical probability space $(\probSp,\sigmaField,\prob)$
with a filtration $\{\sigmaField(t)\}_{0\leq t\leq T}$ satisfying the usual conditions.
The parameter $\beta$ is called the inverse temperature.
It is known that the process $X$ with $\beta=1,2,4$ is obtained as
an eigenvalue process of some matrix-valued Brownian motion
(\cite{Dyson1962a}, \cite{AndersonGuionnetZeitouni2010}, \cite{Katori2015}, \cite{Mehta2004}).
Further, $X$ with $\beta=2$ is obtained as a standard Brownian motion with the non-colliding condition
or Doob's $h$-transform of an absorbing Brownian motion in $\Delta_d$
(\cite{Grabiner1999}, \cite{Biane1995}).
Therefore, this process is studied using various methods.

The existence of density functions of Dyson Brownian motion has been studied in various ways.
Because the Dyson Brownian motion with parameter $\beta=1,2,4$
is obtained as an eigenvalue process of a matrix-valued Brownian motion,
we see the existence of density functions and explicit forms
(see \cite[Theorem~2.5.2]{AndersonGuionnetZeitouni2010}, \cite[Theorem~3.3.1]{Mehta2004}).
For $\beta=2$, we can also derive the density function in the context of Karlin-McGregor formula, Brownian motion with the non-colliding condition and Doob's $h$-transform
(see \cite[Theorem~3.3 and Equation~(3.32)]{Katori2015}).
Meanwhile, as a general framework for an analytical approach to Dyson Brownian motion, a (radial) Dunkl process has been studied.
The Dunkl process is a c\`adl\`ag Markov process with martingale property and its infinitesimal generator is the Dunkl Laplacian, which is a differential operator with a root system in $\RealNum^d$
(for more details, see \cite{GraczykRoslerYor2008}).
Moreover, the semigroup density of the Dunkl process exists and can be express by the normalized spherical Bessel functions (see \cite{Rosler1998}).
The radial part of Dunkl process is a solution of some stochastic differential equation
(see \cite[Corollary 6.6]{GraczykRoslerYor2008}) valued in the fundamental Weyl chamber of associated root system and its semigroup density also exists and have some representation (and therefore a Dyson Brownian motion has a density for all $\beta \geq 1$ because it is a radial Dunkl process with the root system type $A$).
It is worth noting that in the theory of Dunkl process, the diffusion coefficient $\sigma$ must be the identity matrix.

The aim of the present paper is to apply Malliavin calculus
to non-colliding particle systems such as hyperbolic particle systems (see \cite{CepaLepingle2001})
and Dyson Brownian motion with smooth drift.
We prove that under some moment condition (see \aref{ass_1515645743} below),
solutions of non-colliding particle systems
admit continuous density functions with respect to the Lebesgue measure.
We use the result by Florit and Nualart \cite{FloritNualart1995}
and Naganuma \cite{Naganuma2013} to deal with
the singularity of the drift coefficient $(x_i-x_k)^{-1}$.
As with Bessel-type processes,
the inverse moment of $X_i(t)-X_k(t)$ plays an important role in the proof.
We show the integrability by using the Girsanov transformation for non-colliding particle systems,
which was inspired by Yor \cite{Yor1980a}
(this approach can also be found in \cite{Chybiryakov2006}).

\subsection{Main Result}\label{sec_MainResults}
We treat an extension of \eqref{eq_DysonBM} and consider the existence
and continuity of the density function of its solution at time $t>0$.

We consider a constant diffusion coefficient
$\sigma=(\sigma_{ij})_{1\leq i,j\leq d}$ that is an invertible matrix.
Next, we introduce a drift coefficient $f=(f_1,\dots,f_d)^\top \colon\Delta_d\to\RealNum^d$
consisting of a singular part $a$ and a smooth part $b$ as follows.
Let $\alpha=(\alpha_{ik})_{1\leq i,j\leq d}$ be
a symmetric matrix with non-negative components.
For $i>j$ (resp.,\ $i<j$), we set $I_{ij}=(0,\infty)$ (resp.,\ $I_{ij}=(-\infty,0)$).
We define $\phi_{ij}\colon I_{ij}\to I_{ij}$ by $\phi_{ij}(\xi)=\alpha_{ij}/\xi$ for every $i\neq j$
and $a=(a_1,\dots,a_d)^\top\colon\Delta_d\to\RealNum^d$ by
\begin{align}\label{eq_1516869956}
	a_i(x)
	=
		\sum_{k;k\neq i}
			\phi_{ik}(x_i-x_k)
	=
		\sum_{k;k\neq i}
			\frac{\alpha_{ik}}{x_i-x_k}.
\end{align}
Let $b=(b_1,\dots,b_d)^\top\colon\RealNum^d\to\RealNum^d$ be a smooth function
which has bounded derivatives of all orders ($b$ itself need not be bounded).
We set $f_i=a_i+b_i$.

For such coefficients $f$ and $\sigma$,
we consider a solution $X=\{X(t)\}_{0 \leq t \leq T}$ to an SDE
\begin{align}\label{eq_NonColPartSys}
	\left\{
		\begin{aligned}
			dX_i(t)
			&=
				f_i(X(t))\,
				dt
				+
				\sum_{k=1}^d
					\sigma_{ik}\,
					dW_k(t),
			\quad
			i=1,\dots,d,\\
			X(0)
			&=
				\bar{x}
				\in \Delta_d
		\end{aligned}
	\right.
\end{align}
as an extension of \eqref{eq_DysonBM}.
\begin{assumption}
	\phantomsection
	\label{ass_1515645743}
	Let $0<T<\infty$ be fixed.
	\begin{enumerate}
		\item	\label{item_1515735667}
				Existence and uniqueness of a strong solution $X$ to SDE \eqref{eq_NonColPartSys}
				such that $\prob(X(t)\in\Delta_d \text{\ for all\ } 0\leq t\leq T)=1$.
		\item	\label{item_1515740401}
				For some $q>d$, it holds that
				\begin{align}\label{ass_inv_moment}
					\max_{\substack{1\leq i,j\leq d\\i\neq j}}
					\sup_{0\leq t\leq T}
						\expect[|X_i(t)-X_j(t)|^{-6q}]
					<
						\infty.
				\end{align}
	\end{enumerate}
\end{assumption}

The following is our main theorem.
\begin{theorem}\label{thm_1512205183}
	Let $0<t\leq T$.
	Under \aref{ass_1515645743},
	the solution $X(t)$ admits a continuous density function
	with respect to the Lebesgue measure.
\end{theorem}

We prove this theorem by showing that $X(t)$ is locally non-degenerate
in the sense of \cite{FloritNualart1995} and \cite{Naganuma2013}.
For details, see \secref{Sec_Malliavin}.
Note that we do not assume that the diffusion matrix $\sigma$ is identity matrix.
For examples of such SDEs, see \cref{cor_1} below.

By similar arguments, we might show the smoothness under stronger assumption for $q$ in  \aref{ass_1515645743} (\ref{item_1515740401}).
It is known that a Dyson Brownian motion can be considered with $X(0)=\bar{x} \in \overline{\Delta_d}$.
However, for proving \tref{thm_1512205183}, we need the inverse moment condition \eqref{ass_inv_moment} thus the initial value $X(0)=\bar{x}$ is not in the boundary of $\Delta_d$.

Next, we propose and prove a criterion of \aref{ass_1515645743}.
For this purpose, we need additional assumptions on $f=a+b$ and $\sigma$ as follows.
\begin{assumption}\phantomsection
	\label{ass_1516937057}
	\begin{enumerate}
		\item	\label{item_1516957962}
				\begin{enumerate}
					\item	$\sigma=\idMat$, where $\idMat$ is the identity matrix.
					\item	$\alpha_{ik}=\alpha$ in \eqref{eq_1516869956}.
					\item \label{item_c}	$b=\mu+c$, where
							$\mu=(\mu_1,\dots,\mu_d)^\top$ and $c=(c_1,\dots,c_d)^\top\colon\RealNum^d\to\RealNum^d$
							are smooth functions such that
							\begin{itemize}
								\item	$\mu_i$ depends only on the $i$th argument,
										that is, $\mu_i(x)=\tilde{\mu_i}(x_i)$
										for some one-variable function $\tilde{\mu}_i$.
								\item	all derivatives of $\mu_i$ are bounded ($\mu_i$ itself need not be bounded).
								\item	$\mu_k(x)\geq \mu_l(x)$
										for any $x=(x_1,\dots,x_d)\in\Delta_d$ and $k>l$.
								\item	$c_i$ is bounded together with all its derivatives.
							\end{itemize}
				\end{enumerate}
		\item	\label{item_1516937490}
				$\alpha>6d+1/2$.
	\end{enumerate}
\end{assumption}
For example, if $\mu_i(x)=\tilde{\mu}_ix_i$
for a constant $\tilde{\mu}=(\tilde{\mu}_1,\dots,\tilde{\mu}_d)^\top$
with $\tilde{\mu}_k\geq \tilde{\mu}_l$ for $k>l$,
then $\mu$ satisfies all the conditions.
Then, we obtain the following.
\begin{theorem}\label{thm_1516938726}
	Let $0<t\leq T$.
	Under \aref{ass_1516937057},
	the solution $X$ exists and $X(t)$ admits a continuous density function
	with respect to the Lebesgue measure.
\end{theorem}
We comment on \tref{thm_1516938726}.
We show in two steps that \aref{ass_1516937057} implies \aref{ass_1515645743}~\iref{item_1515735667}.
First, we consider the case $c=0$ and $\alpha\geq 1/2$, in which the
results of \cite[Lemma~1]{RogersShi1993}, \cite[Theorem~3.1]{CepaLepingle1997},
and \cite[Theorem~2.2, Corollary~6.2]{GraczykMalecki2014}
ensure that \eqref{eq_NonColPartSys} satisfies \aref{ass_1515645743}~\iref{item_1515735667}.
For general $c$ and $\alpha\geq 1/2$, see \pref{prop_1516938218}.
Note that we can show \pref{prop_1516938218} using only
the existence and uniqueness of solutions of \eqref{eq_NonColPartSys} with $c=0$ and $\alpha=1/2$.
Next, we consider \aref{ass_1515645743}~\iref{item_1515740401}.
To ensure the condition holds, we need \aref{ass_1516937057}~\iref{item_1516937490};
see \pref{prop_1516937693}.
The proofs of \pref[prop_1516938218]{prop_1516937693}
are based on the Girsanov transformation.

We obtain the following for non-identity diffusion matrix $\sigma$, which is not considered in the theory of Dunkl process.
\begin{corollary}\label{cor_1}
Let $0<t\leq T$ and define $\sigma_{d}^2:=\max_{i=1,\ldots,d}\sum_{k=1}^{d}\sigma_{ik}^2$.
We assume $\alpha_{ik}=\alpha >(6d+1)d \sigma_{d}^2/3$ in \eqref{eq_1516869956} and $b=\mu$ satisfies \aref{ass_1516937057} (c).
Then the solution $X$ exists and $X(t)$ admits a continuous density function
with respect to the Lebesgue measure.
\end{corollary}
\begin{proof}
The existence and strong uniqueness follows from \cite[Theorem~3.6]{NgoTagichi2017}.
The assumption on $\alpha$ implies that there exists $q>d$ such that $6q<\frac{3\alpha}{d \sigma_d^2}-1$.
Thus \cite[Lemma 3.4]{NgoTagichi2017} ensures \aref{ass_1515645743} (\ref{item_1515740401}).
\end{proof}

Finally, we note that our framework covers hyperbolic particle systems.
We set $\mu=0$ and define $c=(c_1,\dots,c_d)^\top\colon\RealNum^d\to\RealNum^d$ by
\begin{align*}
	c_i(x)
	=
		\sum_{k;k\neq i}
			\alpha_{ik}\psi(x_i-x_k),
	\quad
	\text{where}
	\quad
	\psi(\xi)
	=
		\begin{cases}
			0,& \xi=0,\\
			\coth\xi-\dfrac{1}\xi,&\xi\neq 0.
		\end{cases}
\end{align*}
Because $\psi$ is smooth and all its derivatives are bounded,
$c$ is smooth and all its derivatives are bounded.
In addition, we obtain
\begin{align*}
	f_i(x)
	=
		\sum_{k;k\neq i}
			\alpha_{ik}
			\left\{
				\frac{1}{x_i-x_k}
				+
				\psi(x_i-x_k)
			\right\}
	=
		\sum_{k;k\neq i}
			\alpha_{ik}
			\coth(x_i-x_k).
\end{align*}

\subsection{Notation and Structure}
In the present paper, we use the following notation.
For $x,y\in\RealNum^n$, $|x|$ and $\langle x,y\rangle$ denote
the Euclidian norm and the Euclidian inner product, respectively.
Let $\Mat{n}{\RealNum}$ be the set of all real square matrices of size $n$
and $\Sym{n}{\RealNum}$ be the set of all real symmetric matrices of size $n$.
We set
$
	|A|
	=
		(
			\sum_{i,j=1}^n A_{ij}^2
		)^{1/2}
$
for $A=(A_{ij})_{1\leq i,j\leq n}\in \Mat{n}{\RealNum}$,
which is a norm $|\cdot|$ on $\Mat{n}{\RealNum}$.
Note that the norm $|\cdot|$ is sub-multiplicative; that is, $|AB|\leq|A||B|$.
For $A\in\Mat{n}{\RealNum}$ and $x\in\RealNum^n$, $A^\top$ and $x^\top$
stand for the transposes of $A$ and $x$, respectively.
The set of $\RealNum^n$-valued continuous functions defined on $[0,T]$
is denoted by $C([0,T];\RealNum^n)$.
For $1<p<\infty$, $L^p([0,T];\RealNum^n)$ stands for
the set of $\RealNum^n$-valued power-$p$ integrable functions on $[0,T]$.
For a smooth $\RealNum^n$-valued function $g=(g_1,\dots,g_n)^\top$ defined on $\RealNum^n$,
we write $g_{ij}'=\partial g_i/\partial x_j$ and $g_{ijk}''=\partial^2 g_i/\partial x_j\partial x_k$.
We also regard $g'=(g_{ij}')_{1\leq i,j\leq n}$ as a $\Mat{n}{\RealNum}$-valued function.
The Kronecker delta is denoted by $\delta_{jk}$.
An indicator function of a set $F$ is denoted by $\indicator{F}$.

This paper is structured as follows.
In \secref{Sec_Preliminaries}, we make some remarks on the drift coefficient $f$
and introduce approximating SDEs of \eqref{eq_NonColPartSys}.
In \secref{Sec_Malliavin}, we apply Malliavin calculus to non-colliding particle systems
by using the approximating SDEs introduced in \secref{Sec_Preliminaries}.
The proof of \tref{thm_1512205183} is deferred to \secref{Sec_Malliavin}.
In \secref{Sec_Example}, we show \tref{thm_1516938726}.
In Appendix~\ref{Appendix}, we study some properties of
matrix-valued ordinary differential equations (ODEs).

\section{Preliminaries}\label{Sec_Preliminaries}

\subsection{Remarks on the Drift Coefficient}
We show some properties of the singular part $a$ of the drift coefficient $f$ from \eqref{eq_NonColPartSys}.
\begin{lemma}\label{lem_1515144318}
	We have the following:
	\begin{enumerate}
		\item	\label{item_1515480353}
				For all $n\in\NaturalNum\cup\{0\}$ and $\xi\in I_{ij}$,
				we have $\frac{d^n\phi_{ij}}{d\xi^n}(\xi)=(-1)^{n+1} \frac{d^n\phi_{ji}}{d\xi^n}(-\xi)$.
		\item	\label{item_1515480796}
				For all $x\in\Delta_d$, $a'(x)$ is symmetric and given by
				\begin{align*}
					a'_{ij}(x)
					=
						\begin{cases}
							\displaystyle
								{
									\sum_{l;l\neq i}
										\phi'_{il}(x_i-x_l)
								},
							&
							i=j,\\
							-\phi'_{ij}(x_i-x_j),
							&
							i\neq j.
						\end{cases}
				\end{align*}
	\end{enumerate}
\end{lemma}
\begin{proof}
	We show Assertion~\eqref{item_1515480353} by induction.
	The case for $n=0$ follows from the definition.
	Assume that the assertion holds for some $n$. Then, we have
	\begin{align*}
		\frac{1}{h}
		\left(
			\frac{d^n\phi_{ij}}{d\xi^n}(\xi+h)
			-
			\frac{d^n\phi_{ij}}{d\xi^n}(\xi)
		\right)
		=
			\frac{(-1)^{n+1}}{-1}
			\frac{1}{-h}
			\left(
				\frac{d^n\phi_{ji}}{d\xi^n}(-\xi-h)
				-
				\frac{d^n\phi_{ji}}{d\xi^n}(-\xi)
			\right).
	\end{align*}
	By letting $h\to 0$, we obtain the assertion for $n+1$.
	The proof is complete.

	Direct computation yields Assertion~\iref{item_1515480796}.
\end{proof}

\begin{lemma}\label{lem_1515650878}
	For any $x\in\Delta_d$ and $y,z\in\RealNum^d$, we have
	\begin{gather}
		\label{eq_1515216539}
		\sum_{i=1}^d
			x_i
			a_i(x)
		=
			\sum_{k,l;k>l}
				(x_k-x_l)
				\phi_{kl}(x_k-x_l)
		=
			\sum_{k,l;k>l}
				\alpha_{kl},\\
		\label{eq_1515216560}
		\langle y,a'(x)y\rangle
		=
			\sum_{k,l;k>l}
				\phi_{kl}'(x_k-x_l)
				(y_k-y_l)^2
		\leq
			0,\\
		\label{eq_1515216548}
		\sum_{j,k=1}^d
			a''_{ijk}(x)y_jz_k
		=
			\sum_{l;l\neq i}
				\phi''_{il}(x_i-x_l)
				(y_i-y_l)
				(z_i-z_l).
	\end{gather}
\end{lemma}
We show \eqref{eq_1515216539} and \eqref{eq_1515216560} by using the following identity:
\begin{align}\label{eq_7439104813}
	\sum_{i=1}^d
		\xi_i
		\sum_{l;l\neq i}
			\eta_{il}
	=
		\sum_{k,l;k>l}
			\{\xi_k \eta_{kl}+\xi_l \eta_{lk}\}
\end{align}
for any $\{\xi_i\}_{1\leq i\leq d}$
and $\{\eta_{ij}\}_{1\leq i,j\leq d,i\neq j}$.

\begin{proof}[Proof of \lref{lem_1515650878}]
	We show \eqref{eq_1515216539}.
	From \eqref{eq_7439104813} and \lref{lem_1515144318}~\iref{item_1515480353}, we have
	\begin{align*}
		\sum_{i=1}^d
			x_i
			a_i(x)
		&=
			\sum_{i=1}^d
				x_i
				\sum_{l;l\neq i}
					\phi_{il}(x_i-x_l)\\
		&=
			\sum_{k,l;k>l}
				\{
					x_k
					\phi_{kl}(x_k-x_l)
					+
					x_l
					\phi_{lk}(x_l-x_k)
				\}\\
		&=
			\sum_{k,l;k>l}
				(x_k-x_l)
				\phi_{kl}(x_k-x_l).
	\end{align*}

	We show \eqref{eq_1515216560}.
	\lref{lem_1515144318}~\iref{item_1515480796} implies
	\begin{align*}
		(a'(x)y)_i
		&=
			\sum_{j=1}^d
				a'_{ij}(x)y_j
		=
			\Bigg(
				\sum_{l;l\neq i}
					\phi'_{il}(x_i-x_l)
			\Bigg)
			y_i
			+
			\sum_{j;j\neq i}
				(-\phi'_{ij}(x_i-x_j))y_j\\
		&=
			\sum_{l;l\neq i}
				\phi'_{il}(x_i-x_l)
				(y_i-y_l).
	\end{align*}
	This expression and \eqref{eq_7439104813} yield
	\begin{align*}
		\langle y,a'(x)y\rangle
		&=
			\sum_{i=1}^d
				y_i
				(a'(x)y)_i
		=
			\sum_{i=1}^d
				y_i
				\Bigg(
					\sum_{l;l\neq i}
						\phi'_{il}(x_i-x_l)
						(y_i-y_l)
				\Bigg)\\
		&=
			\sum_{k,l;k>l}
				\{
					y_k
					\phi'_{kl}(x_k-x_l)
					(y_k-y_l)
					+
					y_l
					\phi'_{lk}(x_l-x_k)
					(y_l-y_k)
				\}\\
		&=
			\sum_{k,l;k>l}
				\phi'_{kl}(x_k-x_l)
				(y_k-y_l)^2.
	\end{align*}

	Direct computation yields \eqref{eq_1515216548}.
	The proof is complete.
\end{proof}

\subsection{Approximating SDEs}
To apply Malliavin calculus to a solution of \eqref{eq_NonColPartSys},
we must consider how to approximate SDEs.
For this purpose, we define for the drift coefficient $f$
a family of approximations $\{f^{(\epsilon)}\}_{0<\epsilon<1}$ on $\RealNum^d$.

First, we define a family of functions $\{\rho_\epsilon\}_{0<\epsilon<1}$
that approximates the function $\rho\colon\RealNum\to\RealNum$
defined by $\rho(\xi)=\xi\indicator{[0,\infty)}(\xi)$ as follows.
We set
\begin{align*}
	\tilde{\lambda}(\xi)
	&=
		\begin{cases}
			e^{-1/\xi},& \xi>0,\\
			0,& \xi\leq 0,\\
		\end{cases}
	&
	\lambda(\xi)
	&=
		\frac{\tilde{\lambda}(\xi)}{\tilde{\lambda}(\xi)+\tilde{\lambda}(1-\xi)},
	&
	\lambda_\epsilon(\xi)
	&=
		\lambda
			\left(
				\frac{\xi}{\epsilon}
			\right)
\end{align*}
and define
\begin{align*}
	\rho_\epsilon(\xi)
	&=
		\epsilon
		+
		\int_{\epsilon}^\xi
			\lambda_\epsilon(\eta)\,
			d\eta.
\end{align*}
Next, for $i\neq j$, we introduce $\{\phi^{(\epsilon)}_{ij}\}_{0<\epsilon<1}$
that approximates $\phi_{ij}$ by
\begin{align*}
	\phi^{(\epsilon)}_{ij}(\xi)
	&=
		\begin{cases}
			\displaystyle
				{
					\phi_{ij}(\epsilon)
					+
						\int_\epsilon^\xi
							\phi_{ij}'(\rho_\epsilon(\eta))\,
							d\eta,
				}
			&
				i>j,\\
			\displaystyle
				{
					\phi_{ij}(-\epsilon)
					+
						\int_{-\epsilon}^\xi
							\phi_{ij}'(-\rho_\epsilon(-\eta))\,
							d\eta,
				}
			&
				i<j.
		\end{cases}
\end{align*}
Finally, we define $a^{(\epsilon)}=(a^{(\epsilon)}_1,\dots,a^{(\epsilon)}_d)^\top$ and $f^{(\epsilon)}=(f^{(\epsilon)}_1,\dots,f^{(\epsilon)}_d)^\top \colon\RealNum^d\to\RealNum^d$ by
\begin{align*}
	a^{(\epsilon)}_i(x)
	&=
		\sum_{k;k\neq i}
			\phi^{(\epsilon)}_{ik}(x_i-x_k),
	&
	f^{(\epsilon)}_i=a^{(\epsilon)}_i+b_i.
\end{align*}

To approximate the solution of \eqref{eq_NonColPartSys},
we consider a solution $X^{(\epsilon)}$ to an SDE
\begin{align}\label{eq_1515462872}
	\left\{
		\begin{aligned}
			dX^{(\epsilon)}_i(t)
			&=
				f^{(\epsilon)}_i(X^{(\epsilon)}(t))\,
				dt
				+
				\sum_{k=1}^d
					\sigma_{ik}\,
					dW_k(t),
			\quad
			i=1,\dots,d,\\
			X^{(\epsilon)}(0)
			&=
				\bar{x}.
		\end{aligned}
	\right.
\end{align}
Note that this SDE admits a unique strong solution
because $f^{(\epsilon)}$ is smooth
and all its derivatives are bounded for all $0<\epsilon<1$.

\begin{remark}
	The functions $\lambda_\epsilon$ and $\rho_\epsilon$ have the following properties.
	\begin{itemize}
		\item	$0\leq\lambda_\epsilon(\xi)\leq 1$ for $\xi\in\RealNum$ and
				$\lambda_\epsilon(\xi)=0$ (resp.\ $1$) for $\xi\leq 0$ (resp.\ $\epsilon \leq \xi$).
		\item	$\rho_\epsilon$ is smooth and non-decreasing.
				We have
				\begin{align*}
					\rho_\epsilon(\xi)
					=
						\begin{cases}
							\displaystyle
								{
									\epsilon
									+
									\int_{\epsilon}^0
										\lambda_\epsilon(\eta)\,
										d\eta,
								}& \xi\leq 0,\\
							\xi,& \epsilon\leq \xi.
						\end{cases}
				\end{align*}
				In addition, we obtain that $\xi\leq\rho_\epsilon(\xi)$ for $0<\xi<\epsilon$.
	\end{itemize}
\end{remark}
\begin{remark}
	Note that $\{\phi^{(\epsilon)}_{ij}\}_{0<\epsilon<1}$ and $\{a^{(\epsilon)}\}_{0<\epsilon<1}$
	are good approximations of $\phi_{ij}$ and $a$, respectively.
	For example, $\phi^{(\epsilon)}_{ij}$ is smooth and non-increasing on $\RealNum$.
	For $i>j$, $\phi^{(\epsilon)}_{ij}$ is expressed as
	\begin{align}\label{eq_1517275514}
		\phi^{(\epsilon)}_{ij}(\xi)
		=
			\begin{cases}
				\displaystyle
					{
						\phi_{ij}(\epsilon)
						+
						\int_{\epsilon}^0
							\phi_{ij}'(\rho_\epsilon(\eta))\,
							d\eta
						+
						\phi_{ij}'(\rho_\epsilon(0))
						\xi
					},&
					\xi\leq 0,\\
				\phi_{ij}(\xi),& \epsilon\leq \xi.
			\end{cases}
	\end{align}
	In particular, $\phi^{(\epsilon)}_{ij}(\xi)\geq 0$ for $\xi\leq 0$.
	In addition, we obtain that
	\begin{align}\label{eq_1517275550}
			\phi^{(\epsilon)}_{ij}(\xi)
			=
				\phi_{ij}(\epsilon)
				+
				\int_{\epsilon}^\xi
					\phi_{ij}'(\rho_\epsilon(\eta))\,
					d\eta
			\leq
				\phi_{ij}(\epsilon)
				+
				\int_{\epsilon}^\xi
					\phi_{ij}'(\eta)\,
					d\eta
			=
				\phi_{ij}(\xi)
		\end{align}
		for $0<\xi<\epsilon$.
		In the estimate, we used the fact that $\xi\leq\rho_\epsilon(\xi)$ for $\xi>0$.
\end{remark}

We obtain the assertions of \lref[lem_1515144318]{lem_1515650878}
in which $\phi_{ij}$ and $a$ are replaced by $\phi^{(\epsilon)}_{ij}$ and $a^{(\epsilon)}$, respectively.
Indeed, we obtain the following.
\begin{lemma}\label{lem_1516867028}
	We have the following.
	\begin{enumerate}
		\item	\label{item_1516866814}
				For all $n\in\NaturalNum\cup\{0\}$ and $\xi\in I_{ij}$,
				we have $\frac{d^n\phi^{(\epsilon)}_{ij}}{d\xi^n}(\xi)=(-1)^{n+1} \frac{d^n\phi^{(\epsilon)}_{ji}}{d\xi^n}(-\xi)$.
		\item	\label{item_1516866818}
				For all $x\in\Delta_d$, $a^{(\epsilon),\prime}(x)$ is symmetric and given by
				\begin{align*}
					a^{(\epsilon),\prime}_{ij}(x)
					=
						\begin{cases}
							\displaystyle
								{
									\sum_{l;l\neq i}
										\phi^{(\epsilon),\prime}_{il}(x_i-x_l)
								},
							&
							i=j,\\
							-\phi^{(\epsilon),\prime}_{ij}(x_i-x_j),
							&
							i\neq j.
						\end{cases}
				\end{align*}
	\end{enumerate}
\end{lemma}

\begin{lemma}\label{lem_1516866857}
	For any $x\in\Delta_d$ and $y,z\in\RealNum^d$, we have
	\begin{gather}
		\label{eq_1516866861}
		\sum_{i=1}^d
			x_i
			a^{(\epsilon)}_i(x)
		=
			\sum_{k,l;k>l}
				(x_k-x_l)
				\phi^{(\epsilon)}_{kl}(x_k-x_l)
		\leq
			\sum_{k,l;k>l}
				\alpha_{kl},\\
		\label{eq_1516866865}
		\langle y,a^{(\epsilon),\prime}(x)y\rangle
		=
			\sum_{k,l;k>l}
				\phi^{(\epsilon),\prime}_{kl}(x_k-x_l)
				(y_k-y_l)^2
		\leq
			0,\\
		\label{eq_1516866870}
		\sum_{j,k=1}^d
			a^{(\epsilon),\prime\prime}_{ijk}(x)y_jz_k
		=
			\sum_{l;l\neq i}
				\phi^{(\epsilon),\prime\prime}_{il}(x_i-x_l)
				(y_i-y_l)
				(z_i-z_l).
	\end{gather}
\end{lemma}
\begin{proof}[Proof of \lref{lem_1516867028}]
	Assertion~\iref{item_1516866814} is shown by induction.
	Indeed, we have
	\begin{align*}
		\phi^{(\epsilon)}_{ij}(\xi)
		&=
			-\phi_{ji}(-\epsilon)
			+
			\int_{\epsilon}^\xi
				\phi_{ji}'(-\rho_\epsilon(\eta))\,
				d\eta\\
		&=
			-\phi_{ji}(-\epsilon)
			-
			\int_{-\epsilon}^{-\xi}
				\phi_{ji}'(-\rho_\epsilon(-\eta))\,
				d\eta\\
		&=
			-\phi^{(\epsilon)}_{ji}(-\xi).
	\end{align*}
	This is Assertion~\iref{item_1516866814} for $n=0$.
	We can prove the assertion for $n\geq 1$ in the same way as for \lref{lem_1515144318}~\iref{item_1515480353}.
	Assertion~\iref{item_1516866818} is a consequence of Assertion~\iref{item_1516866814}.
\end{proof}

\begin{proof}[Proof of \lref{lem_1516866857}]
	From \eqref{eq_1517275514} and \eqref{eq_1517275550},
	we see that $\xi\phi^{(\epsilon)}_{ij}(\xi)\leq \alpha_{ij}$ for $i>j$.
	We can prove the same inequality for $i<j$ in the same way.
	These inequalities imply the inequality part of \eqref{eq_1516866861}.
	The other assertions are easily proved.
\end{proof}

\section{Malliavin Calculus for Non-colliding Particle Systems}\label{Sec_Malliavin}

\subsection{Basics of Malliavin Calculus}
We collect basic results on Malliavin calculus;
see \cite{MatsumotoTaniguchi2017,Shigekawa2004,Nualart2006,IkedaWatanabe1989} for more details.

Let $(\probSp,\sigmaField,\prob)$ be the canonical probability space;
that is, $\probSp=C([0,T];\RealNum^d)$, $\sigmaField$ is the Borel $\sigma$-field on $\probSp$,
and $\prob$ is the Wiener measure.
Set $W(\omega)=\omega$ for $\omega\in\probSp$.
Under the probability measure $\prob$, $W$ is a $d$-dimensional standard Brownian motion.
The Cameron--Martin space associated with $d$-dimensional standard Brownian motion
is denoted by $\CM$;
that is,
$\CM$ consists of
all elements $h\in\probSp$ that have a Radon--Nikodym derivative $\dot{h}$
with respect to the Lebesgue measure and $\dot{h}\in L^2([0,T];\RealNum^d)$.

Let $K$ be a real separable Hilbert space, $k\in\NaturalNum\cup\{0\}$, and $1<p<\infty$.
We denote by $\SobSp{k}{p}(K)$ the Sobolev space
of $K$-valued Wiener functionals defined on $(\probSp,\sigmaField,\prob)$ in the Malliavin sense
with differentiability index $k$ and integrability index $p$.
If there is no risk of confusion, we write simply $\SobSp{k}{p}=\SobSp{k}{p}(\RealNum)$.
We set $L^p(K)=\SobSp{0}{p}(K)$.
For $F\in\SobSp{k}{p}(K)$ and $0\leq l\leq k$,
$D^lF$ denotes the $l$th derivative of $F$.

We use the following sufficient condition
to ensure the existence and continuity of a density function of a Wiener functional.
\begin{proposition}[{\cite[Theorem~2.1]{FloritNualart1995}}, {\cite[Theorem~2.2]{Naganuma2013}}]\label{prop_1511420148}
	Let $1<p,q<\infty$ satisfy $1/p+1/q\leq 1$.
	Suppose $q>d$.
	A functional $F\in\SobSp{1}{p}(\RealNum^d)$ admits a continuous density function on $\RealNum^d$
	if there exists an $\CM^d$-valued Wiener functional $U=(U_1,\dots,U_d)\in\SobSp{1}{q}(\CM^d)$ such that
	$(DF_j,U_k)_\CM=\delta_{jk}$ a.s.\ for any $1\leq j,k\leq d$.
\end{proposition}

\subsection{Malliavin Differentiability}\label{sec_MalliavinDiff}
Let $X$ be a solution of \eqref{eq_NonColPartSys}.
In this subsection, we show Malliavin differentiability of $X(t)$ for every $0\leq t\leq T$ (\pref{prop_1511158825}).
Throughout of this subsection, we suppose that \aref{ass_1515645743}~\iref{item_1515735667} holds.
Before starting our discussion, we comment on the drift coefficients $f=a+b$ and $f^{(\epsilon)}=a^{(\epsilon)}+b$.
Combining \eqref{eq_1515216539} and the linear growth condition of $b$,
we obtain that there exists a positive constant $K$ such that
\begin{align}\label{eq_1517125352}
	2
	\sum_{i=1}^d
		x_i
		f_i(x)
	+
	|\sigma|^2
	=
		2
		\sum_{i=1}^d
			x_i
			a_i(x)
		+
		2
		\sum_{i=1}^d
			x_i
			b_i(x)
		+
		|\sigma|^2
	\leq
		K(1+|x|^2)
\end{align}
for any $x\in\Delta_d$.
In particular, $2\sum_{k,l;k>l}\alpha_{kl}+|\sigma|^2\leq K$ holds.
From \eqref{eq_1516866861} and the linear growth condition of $b$,
the same inequality holds for $f^{(\epsilon)}$,
and the constant $K$ is independent of $\epsilon$.
Since $b'$ is bounded, the constant
\begin{align}\label{eq_1539590587}
	M=\sup_{{x\in\RealNum^d}}|b'(x)|
\end{align}
is finite.

To express the derivative $DX(t)$,
we introduce processes $Y$ and $Z$ as solutions
to the following $\Mat{n}{\RealNum}$-valued stochastic ODEs:
\begin{gather*}
	dY(t)
	=
		+
		f'(X(t))
		Y(t)\,
		dt,
	\qquad
	Y_0
	=
		\idMat,\\
	dZ(t)
	=
		-
		Z(t)
		f'(X(t))\,
		dt,
	\qquad
	Z_0
	=
		\idMat.
\end{gather*}
For auxiliary consideration, we introduce solutions
$Y^{(\epsilon)}$ and $Z^{(\epsilon)}$ to $\Mat{n}{\RealNum}$-valued ODEs
\begin{gather*}
	dY^{(\epsilon)}(t)
	=
		+
		f^{(\epsilon),\prime}(X^{(\epsilon)}(t))
		Y^{(\epsilon)}(t)\,
		dt,
	\qquad
	Y^{(\epsilon)}_0
	=
		\idMat,\\
	%
	dZ^{(\epsilon)}(t)
	=
		-
		Z^{(\epsilon)}(t)
		f^{(\epsilon),\prime}(X^{(\epsilon)}(t))\,
		dt,
	\qquad
	Z^{(\epsilon)}_0
	=
		\idMat.
\end{gather*}
Here, $X^{(\epsilon)}$ is a solution to \eqref{eq_1515462872}.
Note that these four ordinary differential equations
admit unique solutions with probability one
because, for a fixed $\omega\in\probSp$, $f'(X(\bullet))(\omega)$
and $f^{(\epsilon),\prime}(X^{(\epsilon)}(\bullet))(\omega)$
are continuous on $[0,T]$ and thus bounded.
A simple calculation implies that $Z$ and $Z^{(\epsilon)}$
are the inverse matrices of $Y$ and $Y^{(\epsilon)}$, respectively.
The next lemma shows the relationship between $X$, $X^{(\epsilon)}$, and so on.
\begin{lemma}\label{lem_1515739775}
	Under \aref{ass_1515645743}~\iref{item_1515735667},
	there exists a random variable $0<\epsilon_0(T)<1$
	such that $X^{(\epsilon)}(t)=X(t)$, $Y^{(\epsilon)}(t)=Y(t)$, and $Z^{(\epsilon)}(t)=Z(t)$ hold
	for any $0\leq t\leq T$ and $0<\epsilon<\epsilon_0(T)$.
\end{lemma}
\begin{proof}
	For every $0<\epsilon<\min_{2\leq i\leq d}(\bar{x}_i-\bar{x}_{i-1})\wedge 1$,
	we define a stopping time $\tau^{(\epsilon)}$ by
	\begin{align*}
		\tau^{(\epsilon)}
		=
			\inf\{0\leq t\leq T;X(t)\in\Delta_d^{(\epsilon)}\}
			\wedge
			\inf\{0\leq t\leq T;X^{(\epsilon)}(t)\in\Delta_d^{(\epsilon)}\},
	\end{align*}
	where
	$
		\Delta_d^{(\epsilon)}
		=
		\{(x_1,\dots,x_d)^\top\in\Delta_d;\min_{2\leq i\leq d} (x_i-x_{i-1})>\epsilon\}
	$.
	Because $f=f^{(\epsilon)}$ on $\Delta_d^{(\epsilon)}$,
	we see that $X(t)=X^{(\epsilon)}(t)$ holds for any $0\leq t\leq \tau^{(\epsilon)}$.
	\aref{ass_1515645743}~\iref{item_1515735667} ensures
	that a random variable
	$\epsilon_0(T):=\sup\{\epsilon \in (0,1);X_t=X_t^{(\epsilon)},~\forall t \in [0,T]\}$
	exists and it holds that $\tau^{(\epsilon)}=T$ for all $0<\epsilon<\epsilon_0(T)$.
	Hence, $X(t)=X^{(\epsilon)}(t)$ holds for all $0\leq t\leq T$ and $0<\epsilon<\epsilon_0(T)$.
	This implies that $Y(t)=Y^{(\epsilon)}(t)$ and $Z(t)=Z^{(\epsilon)}(t)$ hold
	for any $0\leq t\leq T$ and $0<\epsilon<\epsilon_0(T)$.
\end{proof}
We give estimates of $Y(t)Z(s)$ and $Y^{(\epsilon)}(t)Z^{(\epsilon)}(s)$ for $0\leq s\leq t\leq T$.
\begin{lemma}\label{lem_1516245889}
	For any $0\leq s\leq t\leq T$ and $v\in\RealNum^d$, we have
	\begin{align*}
		|Y(t)Z(s)|
		&\leq
			e^{M(t-s)}
			\sqrt{d},
		&
		|Y(t)Z(s)v|
		\leq
			e^{M(t-s)}|v|,\\
		|Y^{(\epsilon)}(t)Z^{(\epsilon)}(s)|
		&\leq
			e^{M(t-s)}
			\sqrt{d},
		&
		|Y^{(\epsilon)}(t)Z^{(\epsilon)}(s)v|
		\leq
			e^{M(t-s)}|v|.
	\end{align*}
	In particular, the absolute values of the eigenvalues of $Y(t)Z(s)$ and $Y^{(\epsilon)}(t)Z^{(\epsilon)}(s)$
	are less than or equal to $e^{M(t-s)}$.
	Here, $M$ is a non-negative constant defined by \eqref{eq_1539590587}.
\end{lemma}
\begin{proof}
	It follows from \eqref{eq_1515216560} and \eqref{eq_1516866865}
	that the eigenvalues of $a'(x)$ and $a^{(\epsilon),\prime}(x)$
	are less than or equal to zero.
	Recall the boundedness of $b'$ and the definition of $M$.
	Using \pref{prop_1510907899}, we obtain the assertions.
\end{proof}

\begin{lemma}\label{lem_1516244752}
	Let $4\leq p<\infty$.
	We have
	\begin{align*}
		\expect\left[\sup_{0 \leq t \leq T} |X(t)|^p\right]
		&\leq
			C
			(1+|\bar{x}|^p),
		&
		\sup_{0<\epsilon<1}
			\expect\left[\sup_{0 \leq t \leq T}|X^{(\epsilon)}(t)|^p\right]
		&\leq
			C
			(1+|\bar{x}|^p),
	\end{align*}
	where $C$ is a positive constant that depends only on $p$, $T$, and $K$.
\end{lemma}
\begin{proof}
	Because we can give estimates of $\expect[\sup_{0 \leq t \leq T}|X(t)|^p]$ and $\expect[\sup_{0 \leq t \leq T}|X^{(\epsilon)}(t)|^p]$ in the same way,
	we consider $\expect[\sup_{0 \leq t \leq T}|X^{(\epsilon)}(t)|^p]$ only.
	In this proof, $C'$ and $C''$ are positive constants that depend only on $p$, $T$, and $K$.

	Applying It\^o's formula to \eqref{eq_1515462872}, we have
	$
		|X^{(\epsilon)}(t)|^2
		=
			|\bar{x}|^2
			+
			A^{(\epsilon)}(t)
			+
			M^{(\epsilon)}(t)
	$,
	where
	\begin{align*}
		A^{(\epsilon)}(t)
		&=
			\int_0^t
				\bigg\{
					2
						\sum_{i=1}^d
						X^{(\epsilon)}_i(s)
						f^{(\epsilon)}_i(X^{(\epsilon)}(s))
					+
					|\sigma|^2
				\bigg\}\,
				ds,\\
		M^{(\epsilon)}(t)
		&=
			2
			\sum_{i=1}^d
			\sum_{k=1}^d
				\int_0^t
					X^{(\epsilon)}_i(s)
					\sigma_{ik}\,
					dW_k(s).
	\end{align*}
	Recalling \eqref{eq_1517125352} and setting
	$
		\tilde{A}^{(\epsilon)}(t)
		=
			K
			\int_0^t
				(1+|X^{(\epsilon)}(s)|^2)\,
				ds
	$,
	we see that
	$
		A^{(\epsilon)}(t)
		\leq
			\tilde{A}^{(\epsilon)}(t)
	$.
	Hence,
	$
		|X^{(\epsilon)}(t)|^2
		\leq
			|\bar{x}|^2
			+
			\tilde{A}^{(\epsilon)}(t)
			+
			M^{(\epsilon)}(t)
	$,
	which implies
	$
		|X^{(\epsilon)}(t)|^p
		\leq
			3^{p/2-1}
			\{
				|\bar{x}|^p
				+
				|\tilde{A}^{(\epsilon)}(t)|^{p/2}
				+
				|M^{(\epsilon)}(t)|^{p/2}
			\}
	$.

	We estimate the expectations
	of $\sup_{0 \leq t \leq T}|\tilde{A}^{(\epsilon)}(t)|^{p/2}$ and $\sup_{0 \leq t \leq T}|M^{(\epsilon)}(t)|^{p/2}$.
	From the Jensen inequality, we have
	\begin{align*}
		\expect
		\left[
			\sup_{0 \leq u \leq t}|\tilde{A}^{(\epsilon)}(u)|^{p/2}
		\right]
		&\leq
			C'
			\bigg(
				1
				+
				\int_0^t
					\expect
						\left[
							\sup_{0 \leq u \leq s}|X^{(\epsilon)}(u)|^p
						\right]\,
					ds
			\bigg).
	\end{align*}
	Note that
	$
		\langle M^{(\epsilon)}\rangle(t)
		=
			4
			\int_0^t
				|(X^{(\epsilon)}(s))^\top \sigma|^2\,
				ds
		\leq
			4
			\int_0^t
				|X^{(\epsilon)}(s)|^2|\sigma|^2\,
				ds
	$
	holds
	and that
	$
		2\xi^2\eta^2
		\leq
			\xi^4+\eta^4
	$
	for any $\xi,\eta\in\RealNum$.
	Using these inequalities, the Burkholder--Davis--Gundy inequality, and the Jensen inequality, we have
	\begin{align*}
		\expect
			\left[
				\sup_{0 \leq u \leq t}|M^{(\epsilon)}(u)|^{p/2}
			\right]
		\leq
			C_p
			\expect[\langle M^{(\epsilon)} \rangle(t)^{p/4}]
		\leq
			C''
			\bigg(
				1
				+
				\int_0^t
					\expect
						\left[
							\sup_{0 \leq u \leq s} |X^{(\epsilon)}(u)|^p
						\right]\,
					ds
			\bigg).
	\end{align*}
	Here, $C_p$ is a constant that appears in the Burkholder--Davis--Gundy inequality
	and depends only on $p$.

	Combining the above, we obtain
	\begin{align*}
		\expect
			\left[
				\sup_{0 \leq u \leq t}|X^{(\epsilon)}(u)|^p
			\right]
		\leq
			3^{p/2-1}
			\left\{
				|\bar{x}|^p
				+
				(C'+C'')
				\bigg(
					1
					+
					\int_0^t
						\expect
							\left[
								\sup_{0 \leq u \leq s}|X^{(\epsilon)}(u)|^p
							\right]\,
						ds
				\bigg)
			\right\}.
	\end{align*}
	This and Gronwall's inequality imply the assertion.
\end{proof}

From \lref[lem_1516245889]{lem_1516244752},
we obtain the next result on the differentiability of $X^{(\epsilon)}(t)$.
\begin{lemma}\label{lem_1512025560}
	Let $0\leq t\leq T$ and $1<p<\infty$.
	We have $X^{(\epsilon)}(t)\in\SobSp{1}{p}(\RealNum^d)$
	and
	\begin{align}\label{eq_1511160310}
		(DX^{(\epsilon)}_i(t))_n
		=
			\int_0^\cdot
				\indicator{[0,t]}(s)
				(
					Y^{(\epsilon)}(t)
					Z^{(\epsilon)}(s)
					\sigma
				)_{in}\,
				ds.
	\end{align}
	Furthermore, it holds that
	\begin{align}\label{eq_1516244784}
		\sup_{0<\epsilon<1}
			\|X^{(\epsilon)}(t)\|_{\SobSp{1}{p}(\RealNum^d)}
		\leq
			C,
	\end{align}
	where $C$ is a positive constant that depends only on $|\bar{x}|$, $K$, $M$, $T$, and $p$.
\end{lemma}
\begin{proof}
	We have $X^{(\epsilon)}(t)\in\SobSp{1}{p}(\RealNum^d)$ and \eqref{eq_1511160310}
	from the standard theory of Malliavin calculus;
	see \cite[Theorem~5.5.1]{MatsumotoTaniguchi2017}, \cite[Theorem~2.2.1]{Nualart2006}.

	We estimate $\expect[\|DX^{(\epsilon)}(t)\|_{\CM^d}^p]$.
	Because the boundedness of $Y^{(\epsilon)}(t)Z^{(\epsilon)}(s)$ is deduced from \lref{lem_1516245889},
	we obtain
	\begin{align*}
		\|DX^{(\epsilon)}(t)\|_{\CM^d}^2
		=
			\int_0^t
				|Y^{(\epsilon)}(t)Z^{(\epsilon)}(s)\sigma|^2\,
				ds
		\leq
			\int_0^t
				e^{2M(t-s)}
				|\sigma|^2\,
				ds
		=
			\frac{e^{2Mt}-1}{2M}
			|\sigma|^2,
	\end{align*}
	where $(e^{2Mt}-1)/(2M)=t$ for $M=0$.
	This implies
	\begin{align*}
		\expect[\|DX^{(\epsilon)}(t)\|_{\CM^d}^p]
		\leq
			\bigg(
				\frac{e^{2MT}-1}{2M}
				|\sigma|^2
			\bigg)^{p/2}.
	\end{align*}
	Combining this estimate and \lref{lem_1516244752}, we obtain \eqref{eq_1516244784}.
\end{proof}

Then, we obtain the following results on the differentiability of $X(t)$
and the expression of its derivative $DX(t)$.
\begin{proposition}\label{prop_1511158825}
	Let $0\leq t\leq T$ and $1<p<\infty$.
	Under \aref{ass_1515645743}~\iref{item_1515735667},
	we have $X(t)\in\SobSp{1}{p}(\RealNum^d)$
	and
	\begin{align}\label{eq_1515734562}
		(DX_i(t))_n
		=
			\int_0^\cdot
				\indicator{[0,t]}(s)
				(
					Y(t)
					Z(s)
					\sigma
				)_{in}\,
				ds.
	\end{align}
\end{proposition}

\begin{proof}
	\lref[lem_1515739775]{lem_1512025560} imply
	\begin{align*}
		\lim_{\epsilon\downarrow 0}
			X^{(\epsilon)}(t)
		&=
			X(t)
		\quad
		\text{a.s.,}
		&
		\lim_{\epsilon\downarrow 0}
			DX^{(\epsilon)}(t)
		&=
			\text{(RHS of \eqref{eq_1515734562})}
		\quad
		\text{a.s.}
	\end{align*}
	\lref{lem_1516244752} implies the uniform integrability of $\{X^{(\epsilon)}(t)\}_{0<\epsilon<1}$ and $\{DX^{(\epsilon)}(t)\}_{0<\epsilon<1}$.
	Hence,
	\begin{align*}
		\lim_{\epsilon\downarrow 0}
			X^{(\epsilon)}(t)
		&=
			X(t)
		\quad
		\text{in $L^p(\RealNum^d)$,}
		&
		\lim_{\epsilon\downarrow 0}
			DX^{(\epsilon)}(t)
		&=
			\text{(RHS of \eqref{eq_1515734562})}
		\quad
		\text{in $L^p(\CM^d)$}.
	\end{align*}
	Combining these $L^p$-convergences
	with the closability of $D$ (\cite[Proposition~1.2.1]{Nualart2006}, \cite[Corollary~4.14]{Shigekawa2004}),
	we obtain the assertion.
\end{proof}

\subsection{Non-degeneracy}
For every $0<t\leq T$, we show the existence and continuity of the density of a solution $X(t)$ to \eqref{eq_NonColPartSys}.
To prove this assertion, we find an $\CM^d$-valued Wiener functional $U$ that
satisfies the assumption of \pref{prop_1511420148}.
Throughout of this subsection, we suppose that \aref{ass_1515645743} holds.

We set
\begin{align*}
	\gamma
	&=
		\idMat-e^{-(M+1)t}Y(t),
	&
	u_k
	&=
		(u_{1k},\dots,u_{dk})^\top,
	\quad
	k=1,\dots,d,
\end{align*}
where $M$ is a non-negative constant defined by \eqref{eq_1539590587} and
\begin{align}\label{eq_1515996735}
	u_{n k}
	=
		\int_0^\cdot
			\indicator{[0,t]}(s)
			\left(
				\sigma^{-1}
				\{(M+1)\idMat-f'(X(s))\}
			\right)_{n k}
			e^{-(M+1)(t-s)}
			\,
			ds,
	\quad
	n=1,\dots,d.
\end{align}
Here, $\gamma$ depends on $t$; however, we suppress $t$ for notational simplicity.
In what follows, we show that an $\CM^d$-valued Wiener functional $U=(U_1,\dots,U_d)$ defined by
\begin{align*}
	U_j
	&=
		\sum_{k=1}^d
			u_k(\gamma^{-1})_{kj},
	\quad
	j=1,\dots,d,
\end{align*}
satisfies the assumption of \pref{prop_1511420148}.

First, we show that $\gamma$ is invertible.
\begin{lemma}\label{lem_1511946573}
	 Under \aref{ass_1515645743}, $|\det\gamma|\geq (1-e^{-t})^d$.
\end{lemma}

\begin{proof}
	We denote the eigenvalues of $Y(t)$, which may be complex number, by $\lambda_1,\dots,\lambda_d$.
	We then have $|\lambda_i|\leq e^{Mt}$ for all $1\leq i\leq d$ from \lref{lem_1516245889} with $s=0$.
	Hence, $|1-e^{-(M+1)t}\lambda_i|\geq 1-e^{-t}$.
	Noting that
	$
		\det\gamma
		=
		e^{-d(M+1)t}
		\det(e^{(M+1)t}\idMat-Y(t))
	$
	and
	putting $\xi=e^{(M+1)t}$ into the polynomial $\det(\xi\idMat-Y(t))=\prod_{i=1}^d(\xi-\lambda_i)$,
	we obtain
	\begin{align*}
		\det\gamma
		=
			e^{-d(M+1)t}
			\prod_{i=1}^d(e^{(M+1)t}-\lambda_i)
		=
			\prod_{i=1}^d
				(1-e^{-(M+1)t}\lambda_i).
	\end{align*}
	Combining the above, we obtain the assertion.
\end{proof}

Next, we study the differentiability of $Y(t)$ in the Malliavin sense.
\begin{lemma}\label{lem_1511426245}
	Under \aref{ass_1515645743}, $Y(t)\in\SobSp{1}{2q}(\RealNum^{d^2})$.
\end{lemma}
\begin{proof}
	Let us consider a solution $\mathcal{Y}^{(\epsilon)}$ of
	\begin{align*}
		d\mathcal{Y}^{(\epsilon)}(t)
		=
			+
			f^{(\epsilon),\prime}(X(t))
			\mathcal{Y}^{(\epsilon)}(t)\,
			dt,
		\qquad
		\mathcal{Y}^{(\epsilon)}_0
		=
			\idMat.
	\end{align*}
	From \cite[Lemma~2.2.2]{Nualart2006} and the boundedness of $f^{(\epsilon),\prime}$,
	we have $\mathcal{Y}^{(\epsilon)}(t)\in\SobSp{1}{2q}(\RealNum^{d^2})$
	and see that $D\mathcal{Y}^{(\epsilon)}_{im}(t)$ is identified with the process $(\Xi^{(\epsilon)}_{imn}(\cdot,t))_{1\leq n\leq d}$ satisfying
	\begin{multline*}
		\Xi^{(\epsilon)}_{imn}(r,t)\\
		=
			\sum_{j,k=1}^d
			\int_r^t
				f^{(\epsilon),\prime\prime}_{ijk}(X(s))
				\tilde{Y}_{kn}(r,s)
				\mathcal{Y}^{(\epsilon)}_{jm}(s)\,
				ds
			+
			\sum_{j=1}^d
			\int_r^t
				f^{(\epsilon),\prime}_{ij}(X(s))
				\Xi^{(\epsilon)}_{jmn}(r,s)\,
				ds,
	\end{multline*}
	where $\tilde{Y}_{kn}(r,s)=\indicator{[0,s]}(r)(Y(s)Z(r)\sigma)_{kn}$.
	To show the assertion by a similar argument to that used for \pref{prop_1511158825},
	we give uniform estimates of $\expect[|\mathcal{Y}^{(\epsilon)}_{im}(t)|^{2q}]$
	and $\expect[\|D\mathcal{Y}^{(\epsilon)}_{im}(t)\|_\CM^{2q}]$ in $\epsilon$.

	It follows from a similar argument to that in \lref{lem_1516245889}
	that $|\mathcal{Y}^{(\epsilon)}(t)|\leq e^{Mt}\sqrt{d}$.
	In the rest of this proof, we give a uniform estimate of
	$\expect[\|D\mathcal{Y}^{(\epsilon)}_{im}(t)\|_\CM^{2q}]$ in $\epsilon$.
	To this end, we write
	$
		\Xi^{(\epsilon)}_{\bullet mn}(r,t)
		=
			(
				\Xi^{(\epsilon)}_{1mn}(r,t),
				\dots,
				\Xi^{(\epsilon)}_{dmn}(r,t)
			)^\top
	$
	and
	$
		|\Xi^{(\epsilon)}_{\bullet mn}(r,t)|^2
		=
			\sum_{i=1}^d
				\Xi^{(\epsilon)}_{imn}(r,t)^2
	$.
	Note
	\begin{align*}
		\|D\mathcal{Y}^{(\epsilon)}_{im}(t)\|_\CM^2
		\leq
			\sum_{l=1}^d
				\|D\mathcal{Y}^{(\epsilon)}_{lm}(t)\|_\CM^2
		=
			\int_0^t
				\sum_{n=1}^d
					|\Xi^{(\epsilon)}_{\bullet mn}(r,t)|^2\,
				dr.
	\end{align*}
	In the rest of this proof, we estimate $|\Xi^{(\epsilon)}_{\bullet mn}(r,t)|^2$ uniformly in $r$.
	The fundamental theorem of calculus and the fact that $\Xi^{(\epsilon)}_{imn}(r,r)^2=0$ yield
	\begin{multline*}
		|\Xi^{(\epsilon)}_{\bullet mn}(r,t)|^2\\
		=
			2
			\sum_{i=1}^d
				\int_r^t
					\Xi^{(\epsilon)}_{imn}(r,s)\,
					d\Xi^{(\epsilon)}_{imn}(r,s)
		=
			2
			\int_r^t
				g^{(\epsilon)}_{mn}(r,s)\,
				ds
			+
			2
			\int_r^t
				h^{(\epsilon)}_{mn}(r,s)\,
				ds,
	\end{multline*}
	where
	\begin{align}
		\label{eq_1515661143}
		g^{(\epsilon)}_{mn}(r,s)
		&=
			\sum_{i=1}^d
				\Xi^{(\epsilon)}_{imn}(r,s)
				\sum_{j,k=1}^d
					f^{(\epsilon),\prime\prime}_{ijk}(X(s))
					\tilde{Y}_{kn}(r,s)
					\mathcal{Y}^{(\epsilon)}_{jm}(s),\\
		h^{(\epsilon)}_{mn}(r,s)
		&=
			\sum_{i=1}^d
				\Xi^{(\epsilon)}_{imn}(r,s)
				\sum_{j=1}^d
					f^{(\epsilon),\prime}_{ij}(X(s))
					\Xi^{(\epsilon)}_{jmn}(r,s).
	\end{align}
	We define functions $g^{(\epsilon),a}_{mn}(r,\cdot)$ and $g^{(\epsilon),b}_{mn}(r,\cdot)$
	by replacing $f^{(\epsilon)}$ by $a^{(\epsilon)}$ and $b$, respectively, in \eqref{eq_1515661143}.
	We use similar symbols for $h^{(\epsilon)}_{mn}(r,\cdot)$.
	Because $f^{(\epsilon)}=a^{(\epsilon)}+b$,
	we have
	$
		g^{(\epsilon)}_{mn}(r,s)
		=
			g^{(\epsilon),a}_{mn}(r,s)
			+
			g^{(\epsilon),b}_{mn}(r,s)
	$
	and
	$
		h^{(\epsilon)}_{mn}(r,s)
		=
			h^{(\epsilon),a}_{mn}(r,s)
			+
			h^{(\epsilon),b}_{mn}(r,s)
	$.

	We first estimate $g^{(\epsilon),a}_{mn}(r,s)$.
	From \eqref{eq_1516866870} and \eqref{eq_7439104813}, we have
	\begin{align*}
		g^{(\epsilon),a}_{mn}(r,s)
		&=
			\sum_{i=1}^d
				\Xi^{(\epsilon)}_{imn}(r,s)
				\sum_{l;l\neq i}
					\phi^{(\epsilon),\prime\prime}_{il}(X_i(s)-X_l(s))\\
		&\phantom{=}\qquad\qquad
					\times
					\{\tilde{Y}_{in}(r,s)-\tilde{Y}_{ln}(r,s)\}
					\{\mathcal{Y}^{(\epsilon)}_{im}(s)-\mathcal{Y}^{(\epsilon)}_{lm}(s)\}\\
		&=
			\sum_{k,l;k>l}
				\{\Xi^{(\epsilon)}_{kmn}(r,s)-\Xi^{(\epsilon)}_{lmn}(r,s)\}
				\phi^{(\epsilon),\prime\prime}_{kl}(X_k(s)-X_l(s))\\
		&\phantom{=}\qquad\qquad
				\times
				\{\tilde{Y}_{kn}(r,s)-\tilde{Y}_{ln}(r,s)\}
				\{\mathcal{Y}^{(\epsilon)}_{km}(s)-\mathcal{Y}^{(\epsilon)}_{lm}(s)\}.
	\end{align*}
	By combining this expression,
	the inequality $\phi^{(\epsilon),\prime\prime}_{kl}(\xi)\leq 2\alpha_{kl}/\xi^3$ for $\xi>0$,
	and \lref{lem_1516245889},
	we have
	\begin{align*}
		|g^{(\epsilon),a}_{mn}(r,s)|
		\leq
			C_1
			\sum_{k,l;k>l}
				\frac{|\Xi^{(\epsilon)}_{kmn}(r,s)-\Xi^{(\epsilon)}_{lmn}(r,s)|}{|X^{(\epsilon)}_k(s)-X^{(\epsilon)}_l(s)|^3},
	\end{align*}
	where $C_1$ is a positive constant such that
	\begin{align*}
		|
			2\alpha_{kl}
			\{\tilde{Y}_{kn}(r,s)-\tilde{Y}_{ln}(r,s)\}
			\{\mathcal{Y}^{(\epsilon)}_{km}(s)-\mathcal{Y}^{(\epsilon)}_{lm}(s)\}
		|
		\leq
			C_1
	\end{align*}
	for any $k$, $l$, $m$, and $n$, and for $r<s$.
	Hence, Young's inequality implies
	\begin{align*}
		|g^{(\epsilon),a}_{mn}(r,s)|
		&\leq
			C_1
			\sum_{k,l;k>l}
				\frac{1}{2}
				\left\{
					(\Xi^{(\epsilon)}_{kmn}(r,s)-\Xi^{(\epsilon)}_{lmn}(r,s))^2
					+
					\frac{1}{(X_k(s)-X_l(s))^6}
				\right\}\\
		&\leq
			2(d-1)C_1
			|\Xi^{(\epsilon)}_{\bullet mn}(r,s)|^2
			+
			\frac{C_1}{2}
			\sum_{k,l;k>l}
					\frac{1}{(X_k(s)-X_l(s))^6}.
	\end{align*}

	Next, we estimate $g^{(\epsilon),b}_{mn}(r,s)$.
	Noting the boundedness of $b''$ and \lref{lem_1516245889},
	we see that there exists a positive constant $C_2$ that satisfies
	\begin{align*}
		\sum_{i=1}^d
			\Bigg(
				\sum_{j,k=1}^d
					b''_{ijk}(X(s))
					\tilde{Y}_{kn}(r,s)
					\mathcal{Y}^{(\epsilon)}_{jm}(s)
			\Bigg)^2
		\leq
			C_2^2
	\end{align*}
	for any $r\leq s$.
	H\"older's inequality and Young's inequality imply
	\begin{align*}
		|g^{(\epsilon),b}_{mn}(r,s)|
		\leq
			|\Xi^{(\epsilon)}_{\bullet mn}(r,s)|
			C_2
		\leq
			\frac{C_2}{2}
			\{
				1
				+
				|\Xi^{(\epsilon)}_{\bullet mn}(r,s)|^2
			\}.
	\end{align*}

	Finally, we estimate $h^{(\epsilon),a}_{mn}(r,s)$
	and $h^{(\epsilon),b}_{mn}(r,t)$.
	Note that
	\begin{align*}
		h^{(\epsilon),a}_{mn}(r,s)
		&=
			\langle
				\Xi^{(\epsilon)}_{\bullet mn}(r,s),
				a^{(\epsilon),\prime}(X(s))\Xi^{(\epsilon)}_{\bullet mn}(r,s)
			\rangle,\\
		h^{(\epsilon),b}_{mn}(r,s)
		&=
			\langle
				\Xi^{(\epsilon)}_{\bullet mn}(r,s),
				b'(X(s))\Xi^{(\epsilon)}_{\bullet mn}(r,s)
			\rangle.
	\end{align*}
	Hence, by using \eqref{eq_1516866865} and noting the boundedness of $b'$, we have
	\begin{align*}
		h^{(\epsilon),a}_{mn}(r,t)
		&\leq
			0,
		&
		|h^{(\epsilon),b}_{mn}(r,t)|
		&\leq
			M
			|\Xi^{(\epsilon)}_{\bullet mn}(r,s)|^2.
	\end{align*}

	Combining the above, we obtain
	\begin{align*}
		|\Xi^{(\epsilon)}_{\bullet mn}(r,t)|^2
		&\leq
			C_3
			(t-r)
			+
			C_4
			\sum_{k,l;k>l}
				\int_r^t
					\frac{ds}{(X_k(s)-X_l(s))^6}
			+
			C_5
			\int_r^t
				|\Xi^{(\epsilon)}_{\bullet mn}(r,s)|^2\,
				ds.
	\end{align*}
	Therefore, Gronwall's inequality implies
	\begin{align*}
		|\Xi^{(\epsilon)}_{\bullet mn}(r,t)|^2
		&\leq
			\left\{
				C_3
				(t-r)
				+
				C_4
				\sum_{k,l;k>l}
					\int_r^t
						\frac{ds}{(X_k(s)-X_l(s))^6}
			\right\}
			e^{C_5(t-r)}\\
		&\leq
			\left\{
				C_3
				t
				+
				C_4
				\sum_{k,l;k>l}
					\int_0^t
						\frac{ds}{(X_k(s)-X_l(s))^6}
			\right\}
			e^{C_5t}.
	\end{align*}
	This and \aref{ass_1515645743}~\iref{item_1515740401}
	imply that $\expect[\|D\mathcal{Y}^{(\epsilon)}_{im}\|_\CM^q]$ is finite.
\end{proof}

\begin{lemma}\label{lem_1511947871}
	Under \aref{ass_1515645743}, $\gamma^{-1}\in\SobSp{1}{2q}(\RealNum^{d^2})$.
\end{lemma}
\begin{proof}
	We use \cite[Proposition~1.2.3]{Nualart2006} to prove this assertion.
	Recall that $\gamma^{-1}=\Gamma^\top/\det\gamma$,
	where $\Gamma$ is the cofactor matrix of $\gamma$.
	Because all elements of $\gamma$ are bounded and belong to $\SobSp{1}{2q}$
	(see \lref[lem_1516245889]{lem_1511426245}),
	we have $\Gamma\in\SobSp{1}{2q}(\RealNum^{d^2})$ and $\det\gamma\in\SobSp{1}{2q}$.
	\lref{lem_1511946573} yields $1/\det\gamma\in\SobSp{1}{2q}$.
	Hence, the assertion holds.
\end{proof}

\begin{lemma}\label{lem_1511497774}
	Under \aref{ass_1515645743}, $u_k\in\SobSp{1}{2q}(\CM)$ for every $1\leq k\leq d$.
\end{lemma}
\begin{proof}
	Let $A_{n k}(s)$ be the integrand in \eqref{eq_1515996735};
	that is, $u_{n k}=\int_0^\cdot A_{n k}(s)\,ds$.
	By \cite[pp.125--126]{Shigekawa2004}, the assertion holds
	if we have
	\begin{enumerate}
		\item	\label{item_1516083876}
				$
					\Psi(s)
					\in
						\SobSp{1}{2q}(\RealNum)
				$
				for all $0\leq s\leq T$,
		\item	\label{item_1516084702}
				$
					\expect
						\left[
							\left(
								\int_0^T
									|\Psi(s)|^2\,
									ds
							\right)^q
						\right]
					<
						\infty
				$,
		\item	\label{item_1516095030}
				$
					\expect
						\left[
							\left(
								\int_0^T
									\|D\Psi(s)\|_\CM^2\,
									ds
							\right)^q
						\right]
					<
						\infty
				$
	\end{enumerate}
	for $\Psi(s)=A_{n k}(s)$.

	First, we set $\Psi(s)=(X_i(s)-X_j(s))^{-2}$ for $i\neq j$
	and show that Assertions~\iref{item_1516083876},~\iref{item_1516084702}, and~\iref{item_1516095030} hold.
	\aref{ass_1515645743}~\iref{item_1515740401} and \pref{prop_1511158825}	yield Assertion~\iref{item_1516083876}.
	Jensen's inequality and \aref{ass_1515645743}~\iref{item_1515740401} imply Assertion~\iref{item_1516084702}.
	Because $D\Psi(s)$ is identified with
	\begin{align*}
		(-2)
		(X_i(s)-X_j(s))^{-3}
		(\tilde{Y}_{in}(\cdot,s)-\tilde{Y}_{jn}(\cdot,s))_{1\leq n\leq d},
	\end{align*}
	where
	$
		\tilde{Y}_{in}(u,s)
		=
			\indicator{[0,s]}(u)
			(
				Y(s)
				Z(u)
				\sigma
			)_{in}
	$,
	we see that $\|D\Psi(s)\|_\CM^2$ is bounded above by
	\begin{multline*}
		\int_0^T
			|
				(-2)
				(X_i(s)-X_j(s))^{-3}
				(\tilde{Y}_{in}(u,s)-\tilde{Y}_{jn}(u,s))_{1\leq n\leq d}
			|^2\,
			du\\
		\begin{aligned}
			&=
				\int_0^T
					4
					(X_i(s)-X_j(s))^{-6}
					|(\tilde{Y}_{in}(u,s)-\tilde{Y}_{jn}(u,s))_{1\leq n\leq d}|^2\,
					du\\
			&\leq
				16d^2|\sigma|^2e^{2MT}(X_i(s)-X_j(s))^{-6}.
		\end{aligned}
	\end{multline*}
	This estimate and \aref{ass_1515645743}~\iref{item_1515740401} ensure Assertion~\iref{item_1516095030}.

	We conclude this proof by showing
	Assertions~\iref{item_1516083876}, \iref{item_1516084702}, and \iref{item_1516095030}
	for $\Psi(s)=A_{n k}(s)$.
	From the above discussion for $\Psi(s)=(X_i(s)-X_j(s))^{-2}$,
	the assertions are valid for $\Psi(s)=a'_{ij}(X(s))$.
	Because $X(s)\in\SobSp{1}{2q}(\CM)$ and $b'_{ij}$ has bounded derivatives,
	the assertions hold for $\Psi(s)=b'_{ij}(X(s))$.
	Because $f=a+b$, Assertions~\iref{item_1516083876}, \iref{item_1516084702}, and \iref{item_1516095030} hold
	for $\Psi(s)=A_{n k}(s)$.
	The proof is complete.
\end{proof}
We are now in a position to prove our main theorem.
\begin{proof}[Proof of \tref{thm_1512205183}]
	We have $X(t)\in\SobSp{1}{p}$ for $1<p<\infty$ with $1/p+1/q\leq 1$ from \pref{prop_1511158825}.
	From \lref[lem_1511947871]{lem_1511497774},
	we have $\gamma^{-1}\in\SobSp{1}{2q}(\RealNum^{d^2})$ and $u\in\SobSp{1}{2q}(\CM^d)$,
	which implies $U\in\SobSp{1}{q}(\CM^d)$.
	In the rest of this proof, we show $\langle DX_i, U_j \rangle_\CM=\delta_{ij}$.
	\pref{prop_1511158825} and the definition \eqref{eq_1515996735} imply
	\begin{gather*}
		\frac{d(DX_i(t))_n}{ds}(s)
		=
			\indicator{[0,t]}(s)
			\sum_{k,l=1}^d
				Y_{ik}(t)
				Z_{kl}(s)
				\sigma_{ln},\\
		\frac{du_{n j}}{ds}(s)
		=
			\indicator{[0,t]}(s)
			\sum_{l=1}^d
				(\sigma^{-1})_{n l}
				((M+1)\delta_{lj}-f'_{lj}(X(s)))
				e^{-(M+1)(t-s)},
	\end{gather*}
	respectively.
	By summing the product of these terms over $n$,
	we obtain
	\begin{multline*}
		\langle DX_i(t), u_j \rangle_\CM\\
		\begin{aligned}
			&=
				\int_0^t
					\sum_{k,l=1}^d
						Y_{ik}(t)Z_{kl}(s)
						((M+1)\delta_{lj}-f'_{lj}(X(s)))e^{-(M+1)(t-s)}\,
					ds\\
			&=
				\sum_{k=1}^d
					Y_{ik}(t)
					e^{-(M+1)t}
					\int_0^t
						\sum_{l=1}^d
							Z_{kl}(s)
							((M+1)\delta_{lj}-f'_{lj}(X(s)))e^{(M+1)s}\,
						ds\\
			&=
				\sum_{k=1}^d
					Y_{ik}(t)
					e^{-(M+1)t}
					\int_0^t
						\left\{
							Z_{kj}(s)
							\frac{de^{(M+1)s}}{ds}
							+
							\frac{dZ_{kj}}{ds}(s)
							e^{(M+1)s}
						\right\}\,
						ds.
		\end{aligned}
	\end{multline*}
	The integration by parts formula implies
	\begin{multline*}
		\langle DX_i(t), u_j \rangle_\CM\\
		\begin{aligned}
			=
				\sum_{k=1}^d
					Y_{ik}(t)
					e^{-(M+1)t}
					\{Z_{kj}(t)e^{(M+1)t}-\delta_{kj}\}
			=
				\delta_{ij}-e^{-(M+1)t}Y_{ij}(t)
			=
				\gamma_{ij}.
		\end{aligned}
	\end{multline*}
	From this, we have
	\begin{align*}
		\langle DX_i(t), U_j \rangle_\CM
		=
			\sum_{k=1}^d
				\langle DX_i(t), u_k \rangle_\CM
				(\gamma^{-1})_{kj}
		=
			\sum_{k=1}^d
				\gamma_{ik}
				(\gamma^{-1})_{kj}
		=
			\delta_{ij}.
	\end{align*}
	The proof is complete.
\end{proof}

\section{Proof of \tref{thm_1516938726}}\label{Sec_Example}
In this section, we show that \aref{ass_1516937057} implies \aref{ass_1515645743}.
As a result we obtain \tref{thm_1516938726}.
We denote by $X^{\alpha,\mu+c}$ a solution to \eqref{eq_NonColPartSys}
with $\alpha_{ik}=\alpha$, $b=\mu+c$, and $\sigma=\idMat$
and call it a Dyson Brownian motion with a parameter $\alpha$ and a smooth drift $\mu+c$.
The goal of this section is to show the following propositions.

\begin{proposition}\label{prop_1516938218}
	Let \aref{ass_1516937057}~\iref{item_1516957962} and $\alpha\geq 1/2$ be satisfied.
	Then, there exists a unique strong solution $X^{\alpha,\mu+c}$ to \eqref{eq_NonColPartSys}
	such that $\prob(X^{\alpha,\mu+c}(t)\in\Delta_d \text{\ for all\ } 0\leq t\leq T)=1$.
\end{proposition}

\begin{proposition}\label{prop_1516937693}
	Let \aref{ass_1516937057}~\iref{item_1516957962} and $\alpha > 1/2$ be satisfied.
	For any $0 \leq q<\alpha-1/2$, we have
	\begin{align}\label{eq_1517028447}
		\max_{\substack{1\leq i,k\leq d,\\i \neq k}}
		\sup_{0 \leq t \leq T}
			\expect[|X_i^{\alpha,\mu+c}(t)-X_k^{\alpha,\mu+c}(t)|^{-q}]
		<
			\infty.
	\end{align}
\end{proposition}

\tref{thm_1516938726} is a direct consequence of these propositions as follows.
\begin{proof}[Proof of \tref{thm_1516938726}]
	\pref{prop_1516938218} implies \aref{ass_1515645743}~\iref{item_1515735667}.
	From \pref{prop_1516937693}, we have \aref{ass_1515645743}~\iref{item_1515740401}
	under the conditions that
	$c$ is bounded together with the first derivatives and $\alpha>6d+1/2$.
	This completes the proof.
\end{proof}

\subsection{Existence and Uniqueness}
This subsection is devoted to proving \pref{prop_1516938218}.
First, we show the pathwise uniqueness for \eqref{eq_NonColPartSys}.
\begin{lemma}\label{lem_1517027748}
	The pathwise uniqueness of solutions of \eqref{eq_NonColPartSys} holds.
\end{lemma}
\begin{proof}
	Let $(X,W)$ and $(\tilde{X},W)$ be two solutions of \eqref{eq_NonColPartSys}.
	It\^o's formula yields
	\begin{align*}
		|X(t)-\tilde{X}(t)|^2
		=
			2
			\int_0^t
				A(X(s),\tilde{X}(s))\,
				ds
			+
			2
			\int_0^t
				B(X(s),\tilde{X}(s))\,
				ds,
	\end{align*}
	where
	\begin{align*}
		A(x,\tilde{x})
		&=
			\alpha
			\sum_{i=1}^d
				(x_i-\tilde{x}_i)
				(a_i(x)-a_i(\tilde{x})),
		&
		B(x,\tilde{x})
		&=
			\sum_{i=1}^d
				(x_i-\tilde{x}_i)
				(b_i(x)-b_i(\tilde{x})).
	\end{align*}
	From \eqref{eq_7439104813} and $(\xi-\eta)(\xi^{-1}-\eta^{-1})\leq 0$ for $\xi,\eta>0$,
	we have
	\begin{align*}
		A(x,\tilde{x})
		=
			\alpha
			\sum_{k,l;k>l}
				\{(x_k-x_l)-(\tilde{x}_k-\tilde{x}_l)\}
				\left\{
					\frac{1}{x_k-x_l}
					-
					\frac{1}{\tilde{x}_k-\tilde{x}_l}
				\right\}
		\leq
			0.
	\end{align*}
	The Lipschitz continuity of $b$ implies
	$|B(x,\tilde{x})|\leq K|x-\tilde{x}|^2$ for some $K$ that is independent of $x$ and $\tilde{x}$.
	Hence,
	$
		|X(t)-\tilde{X}(t)|^2
		\leq
			2K
			\int_0^t
				|X(s)-\tilde{X}(s)|^2\,
				ds
	$.
	Therefore, from Gronwall's inequality, we conclude the pathwise uniqueness.
\end{proof}

Next, we show the existence of solutions of \eqref{eq_NonColPartSys}.
Recall that there exists a unique strong solution of \eqref{eq_NonColPartSys} for $\alpha\geq 1/2$ and $c=0$
from \cite[Lemma~1]{RogersShi1993},
\cite[Theorem~3.1]{CepaLepingle1997},
and \cite[Theorem~2.2, Corollary~6.2]{GraczykMalecki2014}.
We use this result to show weak existence and uniqueness in law.
The proof is based on the method of \cite[Lemma~4.5]{Yor1980a}
(see also \cite[Proposition~2.1]{PitmanYor1981} and
\cite[Chapter~XI, Exercise~1.22]{RevuzYor1999}),
which is the Girsanov transformation for Bessel processes to restrict its parameter $\alpha$ to $1/2$.
Note that in this proof, we use only the result on
the unique existence of a strong solution of \eqref{eq_NonColPartSys} with $\alpha=1/2$ and $c=0$.

Before starting our discussion, we fix the notation.
Let $\alpha\geq 1/2$.
Set $\nu=\alpha-1/2$ and $h(x)=\prod_{k,l;k>l}(x_k-x_l)$ for $x\in\Delta_d$.
Let $X^{1/2,\mu}$ be a strong solution of \eqref{eq_NonColPartSys} with $\alpha=1/2$ and $c=0$.
We define processes $M=\{M(t)\}_{0\leq t\leq T}$
and $Z=\{Z(t)\}_{0\leq t\leq T}$ used in the Girsanov transformation by
\begin{align*}
	M(t)
	&=
		\sum_{i=1}^{d}
		\sum_{k;k\neq i}
		\int_{0}^{t}
			\frac{dW_i(s)}{X^{1/2,\mu}_i(s)-X^{1/2,\mu}_k(s)},
	&
	Z(t)
	&=
		\exp
			\left(
				\nu M(t)
				-
				\frac{\nu^2}{2} \langle M \rangle(t)
			\right).
\end{align*}
Because $X^{1/2,\mu}$ satisfies
$
	\prob
		(
			X^{1/2,\mu}(t)\in\Delta_d \text{\ for all\ } 0\leq t\leq T
		)
	=
		1
$,
the process $M$ is well-defined and a local martingale.
Although we see that $Z$ is a local martingale
because it is a solution of an SDE
\begin{align}\label{eq_1516957428}
	Z(t)=1+\nu\int_0^t Z(s)\,dM(s),
\end{align}
we can show that $Z$ is a martingale as follows.
\begin{lemma}
	Let \aref{ass_1516937057}~\iref{item_1516957962} and $\alpha\geq 1/2$ be satisfied.
	Then, the process $Z$ is expressed as
	\begin{multline*}
		Z(t)
		=
			\frac{h(X^{1/2,\mu}(t))^\nu}{h(\bar{x})^\nu}
			\exp
				\left(
					-
					\nu
					\sum_{k,l;k>l}
						\int_0^t
							\frac{\mu_k(X^{1/2,\mu}(s))-\mu_l(X^{1/2,\mu}(s))}{X^{1/2,\mu}_k(s)-X^{1/2,\mu}_l(s)}\,
							ds
				\right.\\
				\left.
					-
					\frac{\nu^2}{2}
					\sum_{i=1}^{d}\sum_{k;k \neq i}
					\int_{0}^t
						\frac{ds}{|X_i^{1/2,\mu}(s)-X_k^{1/2,\mu}(s)|^2}
				\right)
	\end{multline*}
	and is a martingale.
\end{lemma}
\begin{proof}
	We set $F(x)=\log h(x)=\sum_{k,l;k>l} \log (x_k-x_l)$ for $x\in\Delta_d$.
	Then the derivatives of $F$ are given by
	\begin{align*}
		\frac{\partial F}{\partial x_i}(x)
		=
			\sum_{k;k \neq i}\frac{1}{x_i-x_k}
		=:
			u_{i1}(x),
		&
		&
		\frac{\partial^2 F}{\partial x_i^2}(x)
		=
			-\sum_{k;k \neq i}\frac{1}{(x_i-x_k)^2}
		=:
			-u_{i2}(x),
	\end{align*}
	for all $i=1,\ldots,d$.
	From \cite[p.252]{AndersonGuionnetZeitouni2010},
	we note that
	\begin{align}\label{eq_u12}
		\sum_{i=1}^{d}
		u_{i1}(x)^2
		=
			\sum_{i=1}^{d}
			u_{i2}(x).
	\end{align}

	Applying Ito's formula and using \eqref{eq_u12}, we have
	$
		F(X^{1/2,\mu}(t))
		=
			F(\bar{x})
			+A(t)
			+M(t)
	$,
	where
	\begin{align*}
		A(t)
		&=
			\sum_{i=1}^{d}
			\int_{0}^{t}
				u_{i1}(X^{1/2,\mu}(s))
				\mu_i(X^{1/2,\mu}(s))\,
				ds.
	\end{align*}
	Hence,
	\begin{align*}
		Z(t)
		=
			\exp
				\left(
					\nu
					\{
						F(X^{1/2,\mu}(t))-F(\bar{x})-A(t)
					\}
					-
					\frac{\nu^2}{2}
					\langle M\rangle(t)
				\right).
	\end{align*}

	Next, we calculate $Z(t)$.
	The definition of $F$ yields
	\begin{align*}
		\exp(
			\nu
			\{
				F(X^{1/2,\mu}(t))-F(\bar{x})
			\}
			)
		=
			\frac{h(X^{1/2,\mu}(t))^\nu}{h(\bar{x})^\nu}.
	\end{align*}
	From \eqref{eq_7439104813}, we have
	\begin{align*}
		\sum_{i=1}^{d}
			u_{i1}(x)
			\mu_i(x)
		=
			\sum_{i=1}^{d}
				\mu_i(x)
				\sum_{k \neq i}\frac{1}{x_i-x_k}
		=
			\sum_{k,l;k>l}
				\frac{\mu_k(x)-\mu_k(x)}{x_k-x_l},
	\end{align*}
	which implies
	\begin{align*}
		A(t)
		=
			\sum_{k,l;k>l}
				\int_0^t
					\frac{\mu_k(X^{1/2,\mu}(s))-\mu_l(X^{1/2,\mu}(s))}{X^{1/2,\mu}_k(s)-X^{1/2,\mu}_l(s)}\,
					ds.
	\end{align*}
	From \eqref{eq_u12}, we obtain
	\begin{align*}
		\langle M\rangle(t)
		=
			\sum_{i=1}^{d}
			\int_{0}^{t}
				u_{i1}(X^{1/2,\mu}(s))^2\,
				ds
		=
			\sum_{i=1}^{d}
			\int_{0}^{t}
				u_{i2}(X^{1/2,\mu}(s))\,
				ds.
	\end{align*}
	Combining the above, we obtain the expression for $Z(t)$.

	Next, we show that $Z$ is a martingale.
	Since $\mu_k(x)\geq \mu_l(x)$ for $k > l$, from the expression of $Z$, we have for any $p>1$ and stopping time $\tau \leq T$,
	\begin{align*}
		|Z(\tau)|^{p}
		\leq
			\sup_{0 \leq t \leq T}
			\frac{h(X^{1/2,\mu}(t))^{p\nu}}{h(\bar{x})^{p \nu}}.
	\end{align*}
	Lemma \ref{lem_1516244752} therefore yields the family of random variables $Z(\tau)$, is uniformly integrable.
	Hence from \cite[Proposition 1.7 in chapter IV]{RevuzYor1999} , $Z$ is a martingale.
\end{proof}

\begin{lemma}\label{lem_1517027501}
	Let \aref{ass_1516937057}~\iref{item_1516957962} and $\alpha\geq 1/2$ be satisfied.
	If we assume that $c=0$,
	then we have the following:
	\begin{enumerate}
		\item	\label{item_1517028015}
				A weak solution of \eqref{eq_NonColPartSys} on $\Delta_d$ for all $0\leq t \leq T$ exists
				and uniqueness in law holds.
		\item	\label{item_1517028029}
				For any measurable function $g\colon C([0,T];\RealNum^d)\to\RealNum$,
				we have
				\begin{align*}
					\expect[g(X^{\alpha,\mu})]
					=
						\expect[g(X^{1/2,\mu})Z(T)]
				\end{align*}
				provided that all the above expectations exist.
	\end{enumerate}
\end{lemma}

\begin{proof}
	We define a new measure $\prob_T(F)=\expect[Z(T)\indicator{F}]$ for $F \in \sigmaField(T)$.
	Then, because $Z$ is a martingale, $\prob_T$ is a probability measure.
	From the Girsanov theorem, the process $B=\{(B_i(t),\dots,B_d(t))\}_{0\leq t\leq T}$ defined by
	\begin{align*}
		B_i(t)
		=
			W_i(t)
			-
			\langle W_i, \nu M \rangle(t)
		=
			W_i(t)
			-
			\nu
			\sum_{k;k\neq i}
			\int_{0}^{t}
				\frac{ds}{X^{1/2,\mu}_i(s)-X^{1/2,\mu}_k(s)}
	\end{align*}
	is a standard Brownian motion on the probability space $(\probSp, \sigmaField(T),\prob_T)$.
	Moreover, we observe that
	\begin{align*}
		X^{1/2,\mu}_i(t)
		&=
			\bar{x}_i
			+
			\int_{0}^{t}
				\left(
					\sum_{k;k\neq i}
						\frac{1/2}{X^{1/2,\mu}_i(s)-X^{1/2,\mu}_k(s)}
						+
						\mu_i(X^{1/2,\mu}(s))
				\right)\,
				ds
			+
			W_i(t)\\
		&=
			\bar{x}_i
			+
			\int_{0}^{t}
				\left(
					\sum_{k;k\neq i}
						\frac{\alpha}{X^{1/2,\mu}_i(s)-X^{1/2,\mu}_k(s)}
					+
					\mu_i(X^{1/2,\mu}(s))
				\right)\,
				ds
			+
			B_i(t),
	\end{align*}
	and thus $(X^{1/2,\mu},B)$ is a weak solution of \eqref{eq_NonColPartSys}
	with $\alpha\geq 1/2$ and $c=0$ on the probability space $(\probSp, \sigmaField(T),\prob_T)$.
	The uniqueness follows from the uniqueness of the case $\alpha=1/2$ and $c=0$.
	This concludes the proof of the statement.
\end{proof}
We generalize \lref{lem_1517027501} as follows.
\begin{lemma}\label{lem_existence}
	Let \aref{ass_1516937057}~\iref{item_1516957962} and $\alpha\geq 1/2$ be satisfied.
	Then we have the following:
	\begin{enumerate}
		\item	\label{item_1516960355}
				A weak solution of \eqref{eq_NonColPartSys} on $\Delta_d$ for all $0\leq t \leq T$ exists
				and uniqueness in law holds.
		\item	\label{item_1516960406}
				For any measurable function $g\colon C([0,T];\RealNum^d)\to\RealNum$,
				we have
				\begin{align*}
					\expect[g(X^{\alpha,\mu+c})]
					&=
						\expect
							\left[
								g(X^{\alpha,\mu})
								\exp
								\left(
									\sum_{i=1}^{d}
									\int_{0}^{T}
										c_i(X^{\alpha,\mu}(s))\,
										dW_i(s)
								\right.
							\right.\\
					&\phantom{=}\qquad\qquad\qquad\qquad\qquad\qquad
							\left.
								\left.
									-
									\frac{1}{2}
									\int_{0}^{T}
										|c(X^{\alpha,\mu}(s))|^2
									ds
								\right)
							\right]
				\end{align*}
				provided that all the above expectations exist.
		\item	\label{item_1516960452}
				Let $p>1$.
				For any measurable function $g\colon C([0,T];\RealNum^d)\to\RealNum$,
				we have
				\begin{align*}
					\expect[|g(X^{\alpha,\mu+c})|]
					\leq
						C
						\expect[|g(X^{\alpha,\mu})|^p]^{1/p},
				\end{align*}
				where $C$ is a positive constant that is independent of $g$.
	\end{enumerate}
\end{lemma}

\begin{proof}
	Let $(X^{\alpha,\mu},W)$ be a weak solution of
	\eqref{eq_NonColPartSys} with $\alpha\geq 1/2$ and $c=0$.
	We set
	$
		\tilde{M}(t)
		=
		\sum_{i=1}^{d}
			\int_0^t
				c_i(X^{\alpha,\mu}(s))\,
				dW_i(s)
	$.
	Because
	$
		\langle \tilde{M} \rangle (t)
		=
			\int_0^t
				|c(X^{\alpha,\mu}(s))|^2\,
				ds
	$
	and	$c$ is bounded,
	the process $\tilde{M}$ satisfies the Novikov condition.
	Hence, for every $q\geq 1$,
	$\{\tilde{Z}_q(t)=\exp(q\tilde{M}(t)-\frac{q^2}{2}\langle \tilde{M} \rangle(t)\}_{0\leq t\leq T}$
	is a martingale starting at $1$.

	Next, we prove Assertion~\iref{item_1516960355}.
	By using the Girsanov transformation
	and the weak existence and uniqueness in law of solutions of \eqref{eq_NonColPartSys} with $c=0$ (\lref{lem_1517027501}),
	we see weak existence and uniqueness in law of solutions of \eqref{eq_NonColPartSys} with any function $c$.
	Note that $\expect[g(X^{\alpha,\mu+c})]=\expect[g(X^{\alpha,\mu})\tilde{Z}_1(T)]$ holds.
	The proof of Assertion~\iref{item_1516960406} is complete.

	Next, we prove Assertion~\iref{item_1516960452}.
	For any $p,q>1$ with $1/p+1/q=1$, H\"older's inequality yields
	$
		\expect[|g(X^{\alpha,\mu+c})|]
		\leq
			\expect[|g(X^{\alpha,\mu})|^p]^{1/p}
			\expect[\tilde{Z}_1(T)^q]^{1/q}
	$.
	Therefore, we need to prove that $\expect[\tilde{Z}_1(T)^q]$ is finite.
	Because
	$
		\tilde{Z}_1(T)^q
		=
			e^{q(q-1) \langle \tilde{M} \rangle(T)/2}
			\tilde{Z}_q(T)
	$,
	$c$ is bounded and $\tilde{Z}_q$ is a martingale starting at $1$,
 	we have
	$
		\expect[\tilde{Z}_1(T)^q]
		\leq
			e^{q(q-1) R^2/2}
	$,
	where $R$ is a positive constant such that $|c(x)|\leq R$ for any $x\in\RealNum^d$.
	This proves Assertion~\iref{item_1516960452}.
\end{proof}

\begin{proof}[Proof of \pref{prop_1516938218}]
	\lref{lem_1517027748} and \lref{lem_existence}~\iref{item_1516960355} imply the assertion.
\end{proof}

\subsection{Inverse Moments}
Next, we prove \pref{prop_1516937693}.
\begin{lemma}\label{lem_1516251703}
	Let \aref{ass_1516937057}~\iref{item_1516957962} and $\alpha\geq 1/2$ be satisfied.
	Assume that $c=0$.
	For any $0 \leq q\leq\alpha-1/2$, we have \eqref{eq_1517028447}.
\end{lemma}
\begin{proof}
	Applying \lref{lem_1517027501}~\iref{item_1517028029}
	with $g(w)=|w_i(t)-w_k(t)|^{-q}$ for $w\in C([0,T];\RealNum^d)$,
	we have
	\begin{align*}
		\expect[|X_i^{\alpha,\mu}(t)-X_k^{\alpha,\mu}(t)|^{-q}]
		&\leq
		\frac{1}{h(\bar{x})^{\alpha-1/2}}
		\expect
		\left[
				\frac{h(X^{1/2,\mu}(t))^{\alpha-1/2}}{|X_i^{1/2,\mu}(t)-X_k^{1/2,\mu}(t)|^q}
		\right]
		\leq
		\frac{C'}{h(\bar{x})^{\alpha-1/2}},
	\end{align*}
	where $C'$ is a positive constant that depends only on
	$\alpha$, $q$, $\bar{x}$, $T$, $K$, and $d$.
	In the last estimate, we used \lref{lem_1516244752}
	because the integrand is reducible
	with respect to $|X^{1/2,\mu}_i(t)-X^{1/2,\mu}_k(t)|^2$.
\end{proof}

\begin{proof}[Proof of \pref{prop_1516937693}]
	We use \lref{lem_existence}~\iref{item_1516960452}
	with $p=(\alpha-1/2)/q>1$ and \lref{lem_1516251703}.
	Then
	\begin{align*}
		\expect[|X_i^{\alpha,\mu+c}(t)-X_k^{\alpha,\mu+c}(t)|^{-q}]
		&\leq
			C
			\expect[|X_i^{\alpha,\mu+c}(t)-X_k^{\alpha,\mu+c}(t)|^{-(\alpha-1/2)}]^{q/(\alpha-1/2)}\\
		&<
			\infty,
	\end{align*}
	which implies the conclusion.
\end{proof}

\appendix
\section{Estimate of Solution of Matrix-valued ODE}\label{Appendix}
We consider a continuous $\Sym{n}{\RealNum}$-valued function $a$ defined on $[0,\infty)$
and denote the eigenvalues of $a(t)$ by $\lambda_1(t),\dots,\lambda_n(t)$.
We assume that a constant $L$ exists such that
$\lambda_i(t)\leq L$ holds for any $1\leq i\leq n$ and $t\geq 0$.
The assumption $a(t)\in\Sym{n}{\RealNum}$ implies that $a(t)$ is diagonalizable;
more precisely, there exists an orthogonal matrix $q(t)$ such that
$a(t)=q(t)^\top \Lambda(t) q(t)$, where $\Lambda(t)=\mathop{\mathrm{diag}}\{\lambda_1(t),\dots,\lambda_n(t)\}$.
We consider a continuous $\Mat{n}{\RealNum}$-valued function $b$ defined on $[0,\infty)$
and assume that a positive constant $M$ exists such that $|b(t)|\leq M$ for any $t$.
Set $f(t)=a(t)+b(t)$.

For a function $f$ satisfying the conditions above,
we consider $\Mat{n}{\RealNum}$-valued ODEs
\begin{gather*}
	\frac{dy}{dt}(t)
	=
		+
		f(t)y(t),
	\qquad
	y(0)
	=
		\idMat,\\
	\frac{dz}{dt}(t)
	=
		-
		z(t)
		f(t),
	\qquad
	z(0)
	=
		\idMat.
\end{gather*}
From $\frac{d(zy)}{dt}(t)=0$, we have $y(t)z(t)=z(t)y(t)=\idMat$.
For every $0\leq s\leq t$,
we set $\tilde{y}(s,t)=y(t)z(s)$.
Then, we have
\begin{gather*}
	\frac{d\tilde{y}}{dt}(s,t)
	=
		+
		f(t)
		\tilde{y}(s,t),
	\qquad
	\tilde{y}(s,s)=\idMat.
\end{gather*}

\begin{proposition}\label{prop_1510907899}
	For any $0\leq s\leq t<\infty$ and $v\in\RealNum^n$, we have
	\begin{gather*}
		|\tilde{y}(s,t)|
		\leq
			e^{(L+M)(t-s)}
			|\idMat|,
		\qquad
		|\tilde{y}(s,t)v|
		\leq
			e^{(L+M)(t-s)}
			|v|.
	\end{gather*}
\end{proposition}
\begin{proof}
	We have the assertion because we can prove that $|\tilde{y}(s,t)|$ and $|\tilde{y}(s,t)v|^2$
	satisfy the assumption of Gronwall's inequality; that is,
	\begin{gather*}
		\frac{1}{2}
		\frac{d}{dt}
		|\tilde{y}(s,t)|^2
		\leq
			(L+M)
			|\tilde{y}(s,t)|^2,
		\qquad
		\frac{1}{2}
		\frac{d}{dt}|\tilde{y}(s,t)v|^2
		\leq
			(L+M)
			|\tilde{y}(s,t)v|^2.
	\end{gather*}

	Because we can prove these two inequalities in a similar way,
	we prove the first inequality only.
	Note that
	\begin{align*}
		\frac{1}{2}
		\frac{d}{dt}|\tilde{y}(s,t)|^2
		&=
			\langle a(t)\tilde{y}(s,t),\tilde{y}(s,t) \rangle
			+
			\langle b(t)\tilde{y}(s,t),\tilde{y}(s,t) \rangle.
	\end{align*}
	We write $\tilde{y}=\tilde{y}(s,t)$ for notational simplicity.
	Because $a(t)$ is diagonalizable by the orthogonal matrix $q(t)$
	and $\Lambda(t)\leq L \idMat$, we have
	\begin{align*}
		\langle a(t)\tilde{y},\tilde{y} \rangle
		=
			\langle q(t)^\top \Lambda(t)q(t)\tilde{y}, \tilde{y} \rangle
		=
			\langle \Lambda(t) q(t) \tilde{y},q(t)\tilde{y} \rangle
		\leq
			L|q(t)\tilde{y}|^2
		=
			L|\tilde{y}|^2.
	\end{align*}
	The boundedness of $b(t)$ implies
	$
		|\langle b(t)\tilde{y},\tilde{y} \rangle|
		\leq
			|b(t)\tilde{y}||\tilde{y}|
		=
			M|\tilde{y}|^2
	$.
	Combining them, we obtain the assertion and complete the proof.
\end{proof}

From \pref{prop_1510907899}, we see that the absolute values of the eigenvalues of $\tilde{y}(s,t)$
are less than or equal to $e^{(L+M)(t-s)}$ as follows.
Let $\lambda$ be an eigenvalue of $\tilde{y}(s,t)$
and $v$ be an eigenvector corresponding to $\lambda$ with $|v|=1$.
Then
$
	|\lambda|
	=
		|\langle	\lambda v,v\rangle|
	=
		|\langle \tilde{y}(s,t) v,v\rangle|
	\leq
		e^{(L+M)(t-s)}
$.

\section*{Acknowledgments}
The authors thank Professor Ryo Takada of Kyushu University for his valuable suggestion.
They are grateful to an anonymous referee for his/her helpful comments.
This work was supported by JSPS KAKENHI Grant Numbers JP17K14202 and JP17H06833.


\begin{thebibliography}{GRY08}

\bibitem[AGZ10]{AndersonGuionnetZeitouni2010}
Greg~W. Anderson, Alice Guionnet, and Ofer Zeitouni.
\newblock {\em An introduction to random matrices}, volume 118 of {\em
  Cambridge Studies in Advanced Mathematics}.
\newblock Cambridge University Press, Cambridge, 2010.

\bibitem[Bia95]{Biane1995}
Philippe Biane.
\newblock Permutation model for semi-circular systems and quantum random walks.
\newblock {\em Pacific J. Math.}, 171(2):373--387, 1995.

\bibitem[Chy06]{Chybiryakov2006}
Oleksandr Chybiryakov.
\newblock {\em Processus de {D}unkl et relation de {L}amperti}.
\newblock PhD thesis, University Paris 6, 2006.

\bibitem[CL97]{CepaLepingle1997}
Emmanuel C{\'e}pa and Dominique L{\'e}pingle.
\newblock Diffusing particles with electrostatic repulsion.
\newblock {\em Probab. Theory Related Fields}, 107(4):429--449, 1997.

\bibitem[CL01]{CepaLepingle2001}
Emmanuel C{\'e}pa and Dominique L{\'e}pingle.
\newblock Brownian particles with electrostatic repulsion on the circle:
  {D}yson's model for unitary random matrices revisited.
\newblock {\em ESAIM Probab. Statist.}, 5:203--224, 2001.

\bibitem[DM11]{DeMarco2011}
Stefano De~Marco.
\newblock Smoothness and asymptotic estimates of densities for {SDE}s with
  locally smooth coefficients and applications to square root-type diffusions.
\newblock {\em Ann. Appl. Probab.}, 21(4):1282--1321, 2011.

\bibitem[Dys62]{Dyson1962a}
Freeman~J. Dyson.
\newblock A {B}rownian-motion model for the eigenvalues of a random matrix.
\newblock {\em J. Mathematical Phys.}, 3:1191--1198, 1962.

\bibitem[FN95]{FloritNualart1995}
Carme Florit and David Nualart.
\newblock {A local criterion for smoothness of densities and application to the
  supremum of the Brownian sheet}.
\newblock {\em Statist. Probab. Lett.}, 22(1):25--31, 1995.

\bibitem[Fri64]{Friedman1964}
Avner Friedman.
\newblock {\em Partial differential equations of parabolic type}.
\newblock Prentice-Hall, Inc., Englewood Cliffs, N.J., 1964.

\bibitem[GM14]{GraczykMalecki2014}
Piotr Graczyk and Jacek Ma{\l}ecki.
\newblock Strong solutions of non-colliding particle systems.
\newblock {\em Electron. J. Probab.}, 19:no. 119, 21, 2014.

\bibitem[Gra99]{Grabiner1999}
David~J. Grabiner.
\newblock Brownian motion in a {W}eyl chamber, non-colliding particles, and
  random matrices.
\newblock {\em Ann. Inst. H. Poincar{\'e} Probab. Statist.}, 35(2):177--204,
  1999.

\bibitem[GRY08]{GraczykRoslerYor2008}
Piotr Graczyk, Margit R{\"o}sler, and Marc Yor, editors.
\newblock {\em Harmonic and Stochastic Analysis of Dunkl Processes}.
\newblock Hermann, 2008.

\bibitem[IW89]{IkedaWatanabe1989}
Nobuyuki Ikeda and Shinzo Watanabe.
\newblock {\em Stochastic differential equations and diffusion processes},
  volume~24 of {\em North-Holland Mathematical Library}.
\newblock North-Holland Publishing Co., Amsterdam-New York; Kodansha, Ltd., Tokyo, second edition, 1989.

\bibitem[Kat15]{Katori2015}
Makoto Katori.
\newblock {\em Bessel processes, {S}chramm-{L}oewner evolution, and the {D}yson model},
{SpringerBriefs in Mathematical Physics},
Springer
2015.

\bibitem[Kus17]{Kusuoka2017}
Seiichiro Kusuoka.
\newblock Continuity and {G}aussian two-sided bounds of the density functions
  of the solutions to path-dependent stochastic differential equations via
  perturbation.
\newblock {\em Stochastic Process. Appl.}, 127(2):359--384, 2017.

\bibitem[Mak16]{Makhlouf2016}
Azmi Makhlouf.
\newblock Representation and {G}aussian bounds for the density of {B}rownian
  motion with random drift.
\newblock {\em Commun. Stoch. Anal.}, 10(2):Article 2, 151--162, 2016.

\bibitem[Meh04]{Mehta2004}
Madan~Lal Mehta.
\newblock {\em Random matrices}, volume 142 of {\em Pure and Applied
  Mathematics (Amsterdam)}.
\newblock Elsevier/Academic Press, Amsterdam, third edition, 2004.

\bibitem[MT17]{MatsumotoTaniguchi2017}
Hiroyuki Matsumoto and Setsuo Taniguchi.
\newblock {\em Stochastic analysis}, volume 159 of {\em Cambridge Studies in
  Advanced Mathematics}.
\newblock Cambridge University Press, Cambridge, {J}apanese edition, 2017.
\newblock It\^o and Malliavin calculus in tandem.

\bibitem[Nag13]{Naganuma2013}
Nobuaki Naganuma.
\newblock Smoothness of densities of generalized locally non-degenerate
  {W}iener functionals.
\newblock {\em Stoch. Anal. Appl.}, 31(4):609--631, 2013.

\bibitem[NT17]{NgoTagichi2017}
Hoang-Long Ngo and Dai Taguchi.
\newblock {\em Semi-implicit Euler-Maruyama approximation for non-colliding particle systems,}
Preprint,
arXiv:1706.10119v2.

\bibitem[Nua06]{Nualart2006}
David Nualart.
\newblock {\em The {M}alliavin calculus and related topics}.
\newblock Probability and its Applications (New York). Springer-Verlag, Berlin,
  second edition, 2006.

\bibitem[Por90]{Portenko1990}
Nikola{\v{\i}}~I. Portenko.
\newblock {\em Generalized diffusion processes}, volume~83 of {\em Translations
  of Mathematical Monographs}.
\newblock American Mathematical Society, Providence, RI, 1990.
\newblock Translated from the Russian by H. H. McFaden.

\bibitem[PY81]{PitmanYor1981}
Jim Pitman and Marc Yor.
\newblock Bessel processes and infinitely divisible laws.
\newblock In {\em Stochastic integrals ({P}roc. {S}ympos., {U}niv. {D}urham,
  {D}urham, 1980)}, volume 851 of {\em Lecture Notes in Math.}, pages 285--370.
  Springer, Berlin, 1981.

\bibitem[R{\"o}s98]{Rosler1998}
Margit R{\"o}sler.
\newblock Generalized {H}ermite polynomials and the heat equation for {D}unkl
  operators.
\newblock {\em Comm. Math. Phys.}, 192(3):519--542, 1998.

\bibitem[RS93]{RogersShi1993}
L.~C.~G. Rogers and Zhan Shi.
\newblock Interacting {B}rownian particles and the {W}igner law.
\newblock {\em Probab. Theory Related Fields}, 95(4):555--570, 1993.

\bibitem[RY99]{RevuzYor1999}
Daniel Revuz and Marc Yor.
\newblock {\em Continuous martingales and {B}rownian motion}, volume 293 of
  {\em Grundlehren der Mathematischen Wissenschaften [Fundamental Principles of
  Mathematical Sciences]}.
\newblock Springer-Verlag, Berlin, third edition, 1999.

\bibitem[Shi04]{Shigekawa2004}
Ichiro Shigekawa.
\newblock {\em Stochastic analysis}, volume 224 of {\em Translations of
  Mathematical Monographs}.
\newblock American Mathematical Society, Providence, RI, 2004.
\newblock Translated from the 1998 Japanese original by the author, Iwanami
  Series in Modern Mathematics.

\bibitem[Yor80]{Yor1980a}
Marc Yor.
\newblock Loi de l'indice du lacet brownien, et distribution de
  {H}artman-{W}atson.
\newblock {\em Z. Wahrsch. Verw. Gebiete}, 53(1):71--95, 1980.

\end{thebibliography}
\end{document}